\documentclass[11pt,a4paper]{article}
\usepackage[english]{babel}
\usepackage{pst-grad} 
\usepackage{pst-plot} 
\usepackage{pstricks}
\usepackage{amsmath,amssymb,mathrsfs,amsthm}
\usepackage[dvips]{graphicx}
\usepackage{epsfig}

\newcommand{\RR}{\mathbb R}

\newcommand{\TT}{\mathbb T}
\newcommand{\pat}{\partial_t}
\newcommand{\pax}{\partial_x}

\newcommand{\jeps}{\mathcal{J}_\epsilon*}

\newcommand{\vv}[1]{\ensuremath{ {\bf #1} }}

\newcommand{\pderiv}[2]{\ensuremath{ \frac{\partial #1}{\partial #2} }}
\newcommand{\cs}{\ensuremath{ c_{\rm s}^2 }}

\newcommand{\rh}{\ensuremath{ \langle\rho\rangle }}
\newcommand{\re}{\text{Re}}
\newcommand{\im}{\text{Im}}

\newcommand{\rhomax}{\ensuremath{ \Vert\rho\Vert_{L^\infty} }}
\newcommand{\rhomm}{\ensuremath{ \Vert\rho_0\Vert_{L^\infty} }}

\newcommand{\cambio}[1]{{ #1}}

\newcounter{comentcount}
\newcounter{teocount}
\newtheorem{lem}{Lemma}
\newtheorem{prop}{Proposition}

\newtheorem{teo}[teocount]{Theorem}

\newenvironment{coment}
{\stepcounter{comentcount} {\bf \tt Remark} {\bf\tt\arabic{comentcount}} }{ }
\title{An approximate treatment of gravitational collapse}
\author{Yago Ascasibar$^{\mbox{{\footnotesize 1}}}$, Rafael Granero-Belinch\'on$^{\mbox{{\footnotesize 2},{\footnotesize 3}}}$ and Jos\'e Manuel Moreno}

\begin{document}

\maketitle 
\footnotetext[1]{Email: \texttt{yago.ascasibar@uam.es}\\
Universidad Aut\'onoma de Madrid\\
Departamento de F\'{i}sica Te\'orica\\
Campus de Cantoblanco,
28049 - Madrid (Spain)}

\footnotetext[2]{Email: \texttt{r.granero@icmat.es}\\
Consejo Superior de Investigaciones Cient\'ificas\\
Instituto de Ciencias Matem\'aticas (CSIC-UAM-UC3M-UCM)\\
C/Nicol\'as Cabrera, 13-15,
Campus de Cantoblanco,
28049 - Madrid (Spain)}

\footnotetext[3]{Current address:\\
Email: \texttt{rgranero@math.ucdavis.edu}\\
Department of Mathematics,\\
University of California, Davis,\\
CA 95616, United States}

\vspace{0.3cm}
\begin{abstract}

This work studies a simplified model of the gravitational instability of an initially homogeneous infinite medium, represented by $\TT^d$, based on the approximation that the mean fluid velocity is always proportional to the local acceleration.
It is shown that, mathematically, this assumption leads to the restricted Patlak-Keller-Segel model considered by J\"ager and Luckhaus or, equivalently, the Smoluchowski equation describing the motion of self-gravitating Brownian particles, coupled to the modified Newtonian potential that is appropriate for an infinite mass distribution.
We discuss some of the fundamental properties of a non-local generalization of this model where the effective pressure force is given by a fractional Laplacian with $0<\alpha<2$, and illustrate them by means of numerical simulations.
Local well-posedness in Sobolev spaces is proven, and we show the smoothing effect of our equation, as well as a \emph{Beale-Kato-Majda}-type criterion in terms of $\rhomax$.
It is also shown that the problem is ill-posed in Sobolev spaces when it is considered backward in time.
Finally, we prove that, in the critical case (one conservative and one dissipative derivative), $\rhomax(t)$ is uniformly bounded in terms of the initial data for sufficiently large pressure forces.

\end{abstract}





\section{Introduction}

Gravitational collapse - the infall of a body under its own gravity - is of paramount importance in Astrophysics.
Stars and galaxies grow from initially small perturbations of a uniform density medium, which is an unstable equilibrium solution of the Newtonian equations of motion.
Upon the action of gravity, overdense regions attract the neighbouring material, accreting all mass within a certain radius of influence into a singular point.
Some of the gravitational energy, though, transforms into heat and/or random motions that oppose collapse.
In most real-life situations (black holes being the exception), the ensuing pressure gradient is able to counteract the gravitational acceleration, and a non-singular equilibrium state is reached.

Mathematically, the evolution of a fluid with mass density $\rho(\vv{r},t)$ and mean local velocity $\vv{v}(\vv{r},t)$ is described by the Euler equations, that amount to the conservation of mass
\begin{equation}
\pat \rho + \nabla\cdot( \rho \vv{v} ) = 0
\end{equation}
and momentum
\begin{equation}
\pat \vv{v} + (\vv{v}\cdot\nabla)\, \vv{v} = -\frac{\nabla P}{\rho} -\nabla U  \label{eq_momentum}
\end{equation}
supplemented by the Poisson equation
\begin{equation}
\Delta U = S_d\, G\, \left( \rho-\rh \right)
\label{eq_Poisson}
\end{equation}
that relates the gravitational potential $U(\vv{r},t)$ to the density distribution of the fluid, where $S_d$ denotes the surface of a unit sphere in $d$ dimensions.
In order to close the system, the initial data in $\rho$ and $\vv{v}$ at some arbitrary time $t_0$ have to be specified, as well as an (effective) barotropic equation of state to relate the gas density and pressure by a function $P(\rho)$ or, equivalently, a sound speed
\begin{equation}
\cs(\rho) = \pderiv{ P(\rho) }{ \rho }
\end{equation}
where $P$ accounts for an isotropic pressure due to the internal random motions of the constituent particles of the fluid, and we will focus on the particular case of a constant sound speed, corresponding to an isothermal equation of state.

The term $\rh$ in equation~\eqref{eq_Poisson} corresponds to the average gas density, and it vanishes for any system with finite mass contained in an infinite space.
Here, we will consider the collapse of an initial perturbation (e.g a protostar) out of a gas cloud that is assumed to be roughly uniform on scales much larger than those relevant to the problem.
Such conditions may be realized by a periodic cubic box of sufficiently large size $L$, i.e. a three-dimensional torus $\TT^3$, as the spatial domain.
In general, for a $d$-dimensional torus with periodicity $L$,
\begin{equation}
\rh = \frac{1}{L^d}\int_{\TT^d}\rho(x)\ dx.
\label{eq_rho_ave}
\end{equation}
It is interesting to note that this average background density is often neglected in Astrophysics, leading to the so-called `Jeans swindle' in the analysis of the stability of an infinite medium \cite{Jeans_1902}.
A deeper discussion of this issue and the validity of equation~\eqref{eq_Poisson} is provided in~\ref{sec_Jeans}.

We will also make the approximation, akin to Darcy's Law in Earth sciences, that the acceleration term that appears in the momentum conservation equation is always proportional to the fluid velocity
\begin{equation}
\pat \vv{v} + (\vv{v}\cdot\nabla)\, \vv{v} \approx \frac{ \vv{v} }{ \tau }
\label{eq_Darcy}
\end{equation}
where $\tau$ is a constant with dimensions of time.
Physically, this relation arises if a strong friction force is present (see e.g. \cite{Che} and references therein).
However, it will also hold \emph{approximately} during the earliest stages of the gravitational collapse of an initially homogeneous medium.
This assumption greatly simplifies the problem, and it is arguably one of the key ingredients of the proposed model.
Its physical motivation and the relation to other, more traditional approaches, are discussed in~\ref{sec_Darcy}.

Substituting this approximation in expression~\eqref{eq_momentum} and taking the divergence, one obtains
$$
\frac{1}{\tau} \nabla\cdot \left( \rho\vv{v} \right) =  - \Delta P - \nabla \left( \rho \nabla U \right),
$$
which, using the continuity and Poisson equations and writing the equation of state in terms of the sound speed, transforms into
$$
\frac{ \pat\rho }{ \tau } = \cs \Delta\rho + \nabla \rho\, \nabla U + S_d\, G \rho \left( \rho - \rh \right).
$$
Since $U$ is fully determined by $\rho$, our model consists on a single, non-linear and non-local, partial differential equation that describes the evolution of the gas density.
Denoting the solution of $\Delta u=f-\langle f \rangle$ by $u=T(f)$, defining $\beta = \frac{4\pi^2\cs}{S_d\,G\rh L^2}$, and choosing the units of time, length, and mass such that $\frac{1}{\tau S_d\,G\rh} = 1$, $L = 2\pi$, and $\rh = 1$, respectively, we finally arrive at
\begin{equation}
\label{eq7.1}
\pat \rho = \beta \Delta\rho + \rho(\rho-1) + \nabla\rho \cdot \nabla T(\rho). 
\end{equation}

This equation is well known, and its mathematical properties have been widely studied in different physical contexts. It can be re-written as
$$\pat \rho=\beta\Delta\rho+\nabla\cdot(\nabla U\rho),\;\; \Delta U=\rho-\rh,$$
that bears obvious resemblances to the vorticity equation for the incompressible Navier-Stokes equations in two dimensions
$$\pat \omega=\nu\Delta\omega+\nabla\cdot(\nabla^{\bot}\psi\omega),\;\; \Delta \psi=\omega,$$
and it has been proposed by J\"ager and Luckhaus \cite{Jaeger_Luckhaus_93} as a simplified version of the classical (parabolic-elliptic) Patlak-Keller-Segel model \cite{Patlak_53,Keller_Segel_70} of chemotaxis in biological systems (see e.g. \cite{Bi6,BCM,blanchet2009critical,blanchet2010functional} and references therein; a recent review of the results concerning this equation can be found in \cite{blanchet2011parabolic}).
The non-periodic (bounded or otherwise) case with $\rh=0$, i.e. the so-called Smoluchowski-Poisson system, has been previously studied in the context of gravitational collapse by P. Biler and co-workers in the mathematical literature (where the problems of existence, conditions for blow-up, radial solutions, stationary solutions, and other qualitative properties have been addressed in \cite{Bi3,Bi4,Bi5}), and by P.~H. Chavanis and co-workers (see e.g. \cite{Cha} and references therein) in the physical literature.
In particular, the stability of an infinite homogeneous medium with non-zero \rh\ and the phase transitions between homogeneous and inhomogeneous states have been discussed in \cite{Chb} and \cite{Chc}, respectively.
L. Corrias, B. Perthame, and H. Zaag have proven that, for small data in $\|\rho_0\|_{L^{d/2}}$, there are global in time weak solutions to equation \eqref{eq7.1} in $d>1$, and blow-up \cambio{can occur} if the smallness condition does not hold (see Theorem 1.1 and 1.2 in \cite{corrias2004global}).
The case of measure-valued weak solutions has been considered by T.Senba and T.Suzuki \cite{senba2002weak}.

In principle, equation~\eqref{eq7.1} with $\beta=0$ would be a model for a perfectly collisionless fluid with negligible random motions (e.g. cold dark matter), whereas $\beta>0$ corresponds to an isothermal ideal gas.
However, the hydrodynamic approximation (i.e. the Euler equation) may also describe a collisionless system, as long as it features an (approximately) isotropic velocity dispersion tensor \cite{BT}.
Here, we would like to consider the generalized equation
\begin{equation}\label{eq7}
\pat \rho = -\beta \Lambda^{\alpha}\rho + \rho(\rho-1) + \nabla\rho \cdot \nabla T(\rho)
\end{equation}
where the fractional Laplacian $\Lambda^\alpha=(-\Delta)^{\alpha/2}$ is defined using Fourier theory,
$$
\widehat{\Lambda^\alpha u}=|\xi|^\alpha\hat{u},
$$
as an alternative intermediate case that interpolates smoothly between cold dark matter and an ordinary fluid.
Physically, this is a non-local adhesion model.
Due to the non-local character of the diffusive operators $\Lambda^\alpha$, the pressure forces now contain information about the total distribution of mass; in other words, the pressure at a given point $x$ does not depend only on $\rho(x)$ but on the overall distribution $\rho(y)$ for all $y$.
Such an equation may develop singularities in a finite time depending on the initial data.
The case $\alpha=1$ is critical in the sense that we have a conservative derivative (the term $\nabla\rho\cdot\nabla U$) and a \emph{`dissipative'} one in the term $-\beta\Lambda\rho$. In the Keller-Segel community, the word `critical' usually refers to the case
in which the scaling invariance of the PDE matches that of the $L^1$ norm, which happens for $\alpha=d$ in $d=\{1,2\}$.

Albeit its physical interpretation is not as straightforward as the Smoluchowski-Poisson or Patlak-Keller-Segel models, the non-local generalization \eqref{eq7} has also received considerable attention during the past years (e.g. \cite{escudero2006fractional,biler2010blowup,bournaveas2010one}).
In particular, D. Li, J. Rodrigo and X. Zhang \cite{LRZ} have established local existence (using the linear semigroup and Duhamel's formula), as well as a continuation criterion derived by contradiction (see our Proposition \ref{SingF}) and some results concerning blow-up in finite time (Proposition \ref{StarF}).
Global existence when the initial data is small in $L^1$ has been recently addressed in \cite{N1} and \cite{N2}.
We will show that, in such a case, $\rhomax(t)$ may be bounded uniformly, and global existence is obtained as a corollary.
We state all our results as a condition on the size of the pressure forces, $\beta$; our sufficient conditions for gravitational stability are \cambio{closely} related to the classical Jeans criterion, $\beta > 1$ \cite{Jeans_1902}, widely applied by the astrophysical community (see~\ref{sec_Jeans}).

The paper is organized as follows:
In Section~\ref{sec_WellPosed}, we show that the problem is locally well posed (forward in time) in Sobolev spaces (Theorem \ref{WellP}) and a \emph{Beale-Kato-Majda}-type criterion involving $\rhomax$ (Proposition \ref{SingF}).
These results were already known (see e.g. Theorems 1.1 and 1.4 in \cite{LRZ}) and are proved here by different techniques.
We also present a result concerning the smoothing effect of our model (Theorem \ref{SE}), due to the linear part when $\beta>0$, and prove that the problem is ill posed backward in time. Although not unexpected, the smoothing effect is, to the best of our knowledge, a new result, and ill-posedness follows as a consequence of the instant analyticity property.
The conditions for gravitational stability, i.e. the decay of small perturbations (Theorem~\ref{Max1_d2}) and fluctuations of arbitrary size (Theorem~\ref{Max}), are investigated in Section~\ref{sec_InfNorm}.
These results are briefly compared with the well-known critical mass phenomenon in higher dimensions.
Using Proposition \ref{SingF}, these theorems imply global in time existence, which we are not aware that has ever been shown for equation~\eqref{eq7} (c.f. Theorem~1.1 in~\cite{corrias2004global}).
A finite-time blow-up result for some initial data in a special class is proved in Theorem~1.10 of~\cite{LRZ}.
These initial data are concentrated near the origin in terms of a parameter $\delta$ such that, for $\delta<<1$, the maximum of the initial datum grows, roughly speaking, like $1/\delta$.
\cambio{Here} we also give a minimum time for blow-up which is the periodic version of Theorem 1.9 in \cite{LRZ}.
Numerical simulations are presented in Section~\ref{sec_Sims}.

\section{Well-posedness}
\label{sec_WellPosed}

\subsection{Well-posedness in Sobolev spaces}

For the sake of mathematical consistency, we show in this section that equation \eqref{eq7} is well posed if $\beta\geq0$ and $0<\alpha\leq2$.
First, let us note that the average density $\rh$ defined in equation \eqref{eq_rho_ave} is conserved during the evolution of the system, and thus it is a constant that only depends on the initial data\footnote{Without loss of generality, one can set $\rh=1$ by an adequate choice of units.}.
We state this claim in the following Lemma:

\begin{lem}[Mass conservation]
Let $\rho(x,t)$ be a classical solution. Then, the total mass is conserved
$$
\int_{\TT^d}\rho(x,t)\ dx=\int_{\TT^d} \rho(x,0)\ dx,
$$
and, as a consequence, $\rh$ is a constant.
\end{lem}
\begin{proof}
We observe that, using $2\pi \hat{f}(0)=\int_\TT f(x)\, dx$, the term 
$$
\int_\TT\Lambda^\alpha u=0.
$$
Now, integrating and using the Divergence Theorem the proof follows.
\end{proof}

To prove the existence and uniqueness of classical solution, we proceed as in \cite{bertozzi-Majda}.
We obtain some \emph{`a priori'} (\emph{i.e.} assuming the existence of classical solution) bounds for the usual norm in the space $H^k(\TT^d)$ and then regularize equation \eqref{eq7} so that all the regularized problems have a classical solution for a uniform time $T$.
To conclude, we use the \emph{`a priori'} bound to show that the solutions to the regularized problem form a Cauchy sequence whose limit is the solution to the original equation.
In order to simplify the notation, we will abbreviate $\rho(x,t) = \rho(x)$, or simply $\rho$, throughout the rest of the paper.

\begin{teo}[Well-posedness]\label{WellP}
Given equation \eqref{eq7} with $\beta\geq0$, $0<\alpha\leq2$, and initial data $0\leq\rho_0(x,0)\in H^k(\TT^d)$, where $k> d/2+2$, there exists a time $\tau=\tau(\rho_0)>0$ and a unique solution $\rho\in C([0,\tau],H^s)\cap L^\infty([0,\tau],H^k(\TT))$, for $s<k$.
Moreover, $\rho\in C^1([0,\tau],C(\TT))\cap C([0,\tau],C^2(\TT))$.
\end{teo}
\begin{proof}
For simplicity, we only work out in detail the case $k=3$ in one spatial dimension, the general case with $d=2,3$ and $k>3$ being analogous.

\emph{(Existence:)} We define the norm in $H^k(\TT)$ as 
$$
\|\cdot\|^2_{H^{k}}=\|\cdot\|_{L^2}^2+\|\pax^k\cdot\|_{L^2}^2.
$$
We multiply the equation by $\rho$ and integrate by parts:
$$
\int_\TT \rho(x)\pat \rho(x)dx=-\beta\int_\TT\rho(x)\Lambda^{\alpha}\rho(x)+ \int_\TT(\rho(x))^ 2(\rho(x)-1)dx+\int_\TT\frac{\pax (\rho(x))^2}{2}\pax Udx
$$
Using Parseval Theorem and H\"{o}lder inequality we get the estimate
\begin{equation}\label{L2}
\frac{d}{dt}\|\rho\|_{L^2}^2\leq -2\beta\|\Lambda^{\alpha/2}\rho\|_{L^2}^2+c_1\rhomax\|\rho\|_{L^2}^2\leq c(\|\rho\|_{H^3}+1)^2\|\rho\|_{H^3}.
\end{equation}
In the same way, for $\|\pax^3\rho\|_{L^2}^2$, we obtain the following bound:
\begin{equation}\label{H3}
\frac{d}{dt}\|\pax^3 \rho\|_{L^2}^2\leq c(\rhomax+\|\pax^2 \rho\|_{L^\infty})\|\rho\|_{H^3}^2\leq c\|\rho\|_{H^3}(\|\rho\|_{H^3}+1)^2.
\end{equation}
Using \eqref{L2} and \eqref{H3}, one finally arrives to
$$
\frac{d}{dt}\|\rho\|_{H^3}\leq c(\|\rho\|_{H^3}+1)^2,
$$
so 
\begin{equation}\label{energybound}
\|\rho\|_{H^3}\leq \frac{\|\rho_0\|_{H^3}}{1-c\|\rho_0\|_{H^3}t}. 
\end{equation}

We only sketch the rest of the proof because it is classical.
In order to show the existence of classical solutions to problem \eqref{eq7}, we consider $\mathcal{J}\in C^\infty_c$, $\mathcal{J}(x)=\mathcal{J}(|x|)$, $\mathcal{J}\geq0$, and $\int_{\RR}\mathcal{J}=1$.
For $\epsilon>0$, we define
\begin{equation}
\mathcal{J}_\epsilon(x)=\frac{1}{\epsilon}\mathcal{J}\left(\frac{x}{\epsilon}\right)
\label{epsi} 
\end{equation}
and consider the regularized problems
$$
\pat \rho^\epsilon=-\beta\jeps\Lambda^\alpha\jeps \rho^\epsilon+\jeps(\jeps\rho^\epsilon(\jeps\rho^\epsilon-1))+
\jeps(\nabla\jeps\rho^\epsilon\cdot\nabla T(\rho^\epsilon)),
$$
where 
$$\Delta T(\rho^\epsilon)= (\jeps\rho^\epsilon-1).$$

In these regularized problems, the total mass is also conserved, and the previous bounds \eqref{L2} and \eqref{H3} hold.
Using classical energy methods (see \cite{bertozzi-Majda}), one may prove that the regularized problems have an unique smooth solution and, using expression \eqref{energybound}, that the sequence $\rho^\epsilon$, indexed in $\epsilon$, is Cauchy in $C([0,\tau],H^s)$, $0\leq s<3$, with $\tau=\tau(\rho_0)\leq \frac{1}{c\|\rho_0\|_{H^3}}$.
Thus, it has a limit $\rho$, which is a classical solution of our problem. In order to show that $\rho\in H^3$ we use that $\rho^\epsilon$ is uniformly bounded in $H^3$ and, due to the strong convergence to $\rho$ in $H^s$ with $s<3$, this implies that $\rho^\epsilon$ are weakly convergent to $\rho$ in $H^3$.

\emph{(Uniqueness:)} Suppose that for the same initial data $\rho_0$ we have two classical solutions $\rho_1,\rho_2\in C([0,\tau],H^3(\TT^d))$. Then the equation for $\rho=\rho_2-\rho_1$ is
$$
\pat \rho=-\beta \Lambda^\alpha \rho+(\rho_2(\rho_2-1)-\rho_1(\rho_1-1))+\nabla \rho_2\cdot\nabla T(\rho_2)-\nabla\rho_1\cdot\nabla T(\rho_1),
$$
so
$$
\frac{1}{2}\frac{d}{dt}\|\rho\|_{L^2}^2\leq -\beta\|\Lambda^{\alpha/2}\rho\|_{L^2}^2+\|\rho\|_{L^2}^2(\|\rho_2\|_{L^\infty}+\|\rho_1\|_{L^\infty})+\int_{\TT^d}\rho\nabla\rho_1\nabla(T(\rho_2)-T(\rho_1))dx.
$$
We observe that the equation for $T(\rho_2)-T(\rho_1)$ is
$$
\Delta (T(\rho_2)-T(\rho_1))=\rho.
$$
Using Poincar\'e inequality we have that 
$$\|\nabla (T(\rho_2)-T(\rho_1))\|_{L^2}\leq c\|\rho\|_{L^2}.$$
Then
$$
\int_{\TT^d}\rho\nabla\rho_1\nabla(T(\rho_2)-T(\rho_1))dx\leq \|\nabla \rho_1\|_{L^\infty}\|\rho\|_{L^2}\|\nabla (T(\rho_2)-T(\rho_1))\|_{L^2}\leq c\|\nabla \rho_1\|_{L^\infty} \|\rho\|_{L^2}^2.
$$
Putting all together we obtain
$$
\frac{d}{dt}\|\rho\|_{L^2}^2\leq c(\rho_2,\rho_1)\|\rho\|_{L^2}^2.
$$ 
Applying Gronwall inequality we conclude the uniqueness.
\end{proof}

We give an alternative proof of a \emph{Beale-Kato-Majda}-type criterion characterising the possible singularities (see \cite{LRZ}): 
\begin{prop}[Singularity formation]\label{SingF}
Suppose that we have
$$ 
\int_0^{t^+}\rhomax<\infty,
$$
Then the solution exists upon time $t^++\sigma$ for a sufficiently small $\sigma$. 
\end{prop}
\begin{proof}
To clarify the exposition, we consider the case with $d=1$ first. To prove the result we will assume that 
\begin{equation}\label{boundBKM}
\int_0^{t^+}\|\rho\|_{L^\infty}(s)ds=M<\infty
\end{equation}

Using \eqref{L2} and \eqref{H3} we have 
$$
\frac{d}{dt}\|\rho\|_{H^3}\leq c(\rhomax+\|\pax^2\rho\|_{L^\infty}+1)\|\rho\|_{H^3},
$$
so, using Gronwall inequality we conclude that
$$
\|\rho\|_{H^3}(t)\leq \|\rho_0\|_{H^3}e^{c\int_0^t\rhomax(s)+\|\pax^2\rho\|_{L^\infty}(s)ds+t}.
$$
Now we follow the technique exposed in \cite{cor2}: we use Rademacher Theorem for $\pax^2\rho(x_t)=\|\pax^2\rho\|_{L^\infty}(t).$ We can do that if the solution is sufficiently smooth. The evolution of this quantity is given by
$$
\pat \pax^2\rho(x_t)=-\beta\Lambda^\alpha\pax^2\rho(x_t)+\pax^2\rho(x_t)(3(\rho(x_t)-1)+\rho(x_t))+3(\pax \rho(x_t))^2.
$$
We use the generalized Landau inequality (see \cite{MS})
$$
|\pax\rho (x)|^2\leq 2\rho(x)\|\pax^2 \rho\|_{L^\infty}.
$$
Due to this pointwise inequality we observe that the quadratic terms $(\pax \rho(x_t))^2$ are linear in $\pax^2 \rho(x_t)$. Thus we obtain
$$
\pat\pax^2\rho(x_t)\leq 10\|\rho\|_{L^\infty}(t)\pax^2\rho(x_t)\Rightarrow \|\pax^2\rho\|_{L^\infty}(t)\leq\|\pax^2\rho_0\|_{L^\infty}e^{10\int_0^t10\|\rho\|_{L^\infty}(r)dr},
$$
and we conclude the result in the one dimensional case.

Now we consider the case $d=2$. The first step is to obtain a bound on $\|\Delta \rho\|_{L^2}$. We multiply the equation by $\Delta^2\rho$ and integrate in space. The dissipative terms are negative so we neglect it. The reaction term contributes with
$$
\int_{\TT^2}\Delta\rho(x)\Delta(\rho(x)(\rho(x)-1))dx\leq c\|\Delta \rho\|^2_{L^2}\|\rho\|_{L^\infty}+\|\Delta \rho\|_{L^2}\||\nabla \rho|^2\|_{L^2},
$$
and using a classical Gagliardo-Nirenberg interpolation inequality we get
$$
\||\nabla\rho|^2\|_{L^2}\leq c\|\nabla\rho\|_{L^4}^2\leq c\|\Delta\rho\|_{L^2}\|\rho-1\|_{L^\infty}.
$$
Thus the contribution of the reaction term can be bounded as
\begin{equation}\label{lap2}
\int_{\TT^2}\Delta\rho(x)\Delta(\rho(x)(\rho(x)-1))dx\leq c\|\Delta \rho\|^2_{L^2}\|\rho\|_{L^\infty}.
\end{equation}
The contribution of the transport term is 
$$
\int_{\TT^2}\Delta^ 2\rho(x)\nabla\rho(x)\cdot\nabla U(x)dx=J_1+J_2+J_3+J_4+J_5,
$$
where
$$
J_1=\int_{\TT^2}\Delta\rho(x)\nabla U(x)\cdot \nabla\Delta\rho(x)dx\leq c\|\Delta\rho\|_{L^2}^2\|\rho\|_{L^\infty},
$$
$$
J_2=\int_{\TT^2}\Delta\rho(x)|\nabla\rho(x)|^2dx\leq c\|\Delta\rho\|_{L^2}\||\nabla\rho|^2\|_{L^2}\leq c\|\Delta\rho\|_{L^ 2}^2\|\rho\|_{L^ \infty},
$$
$$
J_3=4\int_{\TT^2}\Delta\rho(x)\partial^2_{x_2}\rho(x)\partial^2_{x_1} U(x)dx\leq c\|\Delta\rho\|_{L^ 2}^2\|\partial_{x_1}^2 U\|_{L^ \infty},
$$
$$
J_4=2\int_{\TT^2}(\Delta\rho(x))^2\partial^2_{x_1} U(x)dx\leq c\|\Delta\rho\|_{L^ 2}^2\|\partial_{x_1}^2 U\|_{L^ \infty},
$$
$$
J_5=2\int_{\TT^2}\Delta\rho(x)\partial_{x_2}^2\rho(x)(\rho(x)-1)dx\leq c\|\Delta\rho\|^2_{L^ 2}\|\rho\|_{L^ \infty},
$$
and, putting all together we obtain
$$
\frac{d}{dt}\|\Delta\rho\|_{L^2}\leq c\|\Delta\rho\|_{L^2}(\|\rho\|_{L^\infty}+\|\partial_{x_1}^2 U\|_{L^\infty}).
$$
Due to Gronwall inequality we obtain
\begin{equation}\label{eqh2}
\|\rho\|_{H^2}(t)\leq c(M)\|\rho_0\|_{H^2}\exp\left(\int_0^t\|\partial_{x_1}^2 U\|_{L^\infty}(s)+\|\partial_{x_1}\partial_{x_2} U\|_{L^\infty}(s)ds\right).
\end{equation}
We use a classical inequality (see 3.2b and Appendix 1 in \cite{kato1986well} and \cite{beale1984remarks})
$$
\|V\|_{L^\infty}\leq c(\|\zeta\|_{L^p}+\|\zeta\|_{L^\infty}\max\{1,\log(\|\zeta\|_{W^{s-1,p}}/\|\zeta\|_{L^\infty}))\},\;\;1<p<\infty,s>1+2/p
$$
where $V(\zeta)$ is a singular integral operator of Calderon-Zygmund type of $\zeta$.
In our case we have 
\begin{multline}\label{CZine}
\|\partial_{x_i}\partial_{x_j} U\|_{L^\infty}\leq c(\|\rho-1\|_{L^2}+\|\rho-1\|_{L^\infty}\max\{1,\log(\|\rho-1\|_{H^{2}}/\|\rho-1\|_{L^\infty})\})\\
\leq c(M,\|\rho_0\|_{L^2})(1+\|\rho\|_{L^\infty}\max\{1,\log(c\|\rho\|_{H^{2}})\}).\hspace{1.5cm}
\end{multline}
Inserting in \eqref{eqh2} we have
\begin{equation*}
\|\rho\|_{H^2}(t)\leq c(M)\|\rho_0\|_{H^2}\exp\left(c(M,\|\rho_0\|_{L^2})\int_0^t(1+\|\rho\|_{L^\infty}(s)\log(c\|\rho\|_{H^{2}}(s)+e)ds\right)
\end{equation*}
Now let $y(t)=c\|\rho\|_{H^2}(t)+e$. We have
\begin{equation*}
y(t)\leq c(M)y(0)\exp\left(c(M,\|\rho_0\|_{L^2})\int_0^t(1+\|\rho\|_{L^\infty}(s)\log(y(s))ds\right),
\end{equation*}
and write $z(t)=\log(y(t))$. We get
\begin{equation*}
z(t)\leq c(M)+z(0)+c(M,\|\rho_0\|_{L^2})t+\int_0^t \|\rho\|_{L^\infty}(s)z(s)ds,
\end{equation*}
and due to the integral version of the Gronwall inequality we get
\begin{equation}\label{BRT}
z(t)\leq c(M,\|\rho_0\|_{H^2})(1+t).
\end{equation}
Thus, we conclude that the bound \eqref{boundBKM} implies $\rho(t)\in H^2$. To obtain higher regularity we can do the same. Indeed, taking derivatives and multiplying by the correct terms we obtain
\begin{multline*}
\|\rho\|_{H^4}\leq \|\rho_0\|_{H^4}\exp\left(\int_0^t\|\rho\|_{L^\infty}(s)+\|\Delta \rho\|_{L^2}(s)+\|\partial^2_{x_1} U\|_{L^\infty}(s)+\|\partial_{x_1}\partial_{x_2} U\|_{L^\infty}(s)\right)ds,
\end{multline*}
where we used the following classical interpolation inequalities
$$
\|Df\|_{L^4}^2\leq c\|D^2 f\|_{L^2}\|f\|_{L^\infty},\text{ and } \|D^2f\|_{L^4}^2\leq \|D^2f\|_{L^2}\|D^2f\|_{L^\infty}.
$$
Using the previous bounds \eqref{boundBKM} \eqref{CZine} and \eqref{BRT} we conclude the result.
\end{proof}

\subsection{Smoothing effect}

\begin{figure}[ht]\begin{center}
\includegraphics[scale=0.5]{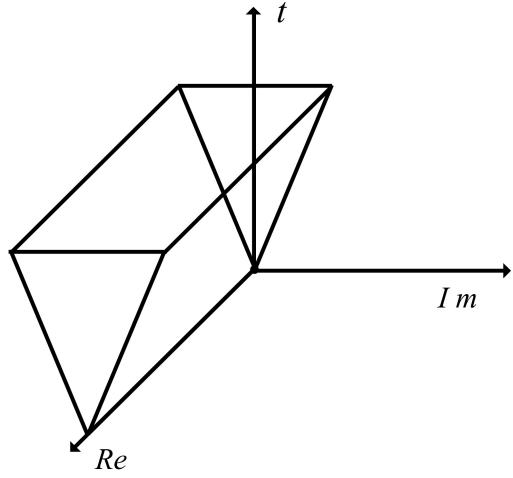}
\end{center}
\caption{The strip of analyticity for $\rho.$}
\label{complexstrip}
\end{figure}

Our proof follows the steps presented in \cite{ccfgl}. The proof is based in \emph{a priori} estimates for the complex extension of the function $\rho$ on the complex strip $S=\{x+i\xi, x\in\TT, |\xi|<kt\}$, for certain constant $k>0$ (see Figure~\ref{complexstrip}). We define
$$
\|\rho\|^2_{L^2(S)}=\int_\TT |\rho(x+ikt)|^2dx+\int_\TT|\rho(x-ikt)|^2dx,\;\;\|\rho\|_{H^3(S)}^2=\|\rho\|_{L^2(S)}^2+\|\pax^3\rho\|^2_{L^2(S)}.
$$
For the complex extension the equation is
\begin{equation}\label{eq7IA}
\pat \rho(x+i\xi) = -\beta \Lambda^{\alpha}\rho(x+i\xi) + \rho(x+i\xi)(\rho(x+i\xi)-1) + \nabla\rho(x+i\xi) \cdot \nabla T(\rho)(x+i\xi)
\end{equation}

\begin{teo}[Smoothing effect]\label{SE}
Let us consider equation \eqref{eq7} with $\alpha\geq1$ with $\rho_0\in H^3$ as initial data and $\TT$ as spatial domain. Then the classical solution $\rho$ (which exists, at least locally in time, because of Theorem \ref{WellP}) continues analitically into the strip $S$ for times $0<t\leq \tau(\rho_0)$.
\end{teo}
\begin{proof}
For the sake of brevity we work with both terms $\pm ikt$ at the same time. Firstly we consider the critical case $\alpha=1$. We start with the $L^2(S)$ norm:
$$
\frac{1}{2}\frac{d}{dt}\int_\TT|\rho(x\pm ikt)|^2dx=\re\int_\TT\bar{\rho}(x\pm ikt)(\pat \rho(x\pm ikt)\pm ik\pax\rho(x\pm ikt))dx.
$$
Using Plancherel Theorem we have
$$
-\beta\int_\TT\bar{\rho}(x\pm ikt)\Lambda\rho(x\pm ikt)dx=-\beta\int_\TT|\Lambda^{1/2}\rho(x\pm ikt)|^2dx\leq0.
$$
The reaction term can be bounded as follows:
$$
\re\int_\TT\bar{\rho}(x\pm ikt)\rho(x\pm ikt)(\rho(x\pm ikt)-1)dx\leq(\|\rho\|_{L^\infty(S)}+1)\|\rho\|_{L^2(S)}^2.
$$ 
We have
\begin{multline*}
I_1=\re\int_\TT\bar{\rho}(x\pm ikt)\pax\rho(x\pm ikt)\pax U(x\pm ikt)dx\\
=\re\int_\TT\pax U(x\pm ikt)\frac{1}{2}\pax |\rho(x\pm ikt)|^2dx\\
+\re\; i\int_\TT\pax U(x\pm ikt)\re\rho(x\pm ikt)\pax\im\rho(x\pm ikt)dx\\
-\re\; i\int_\TT\pax U(x\pm ikt)\im\rho(x\pm ikt)\pax\re\rho(x\pm ikt)dx.
\end{multline*}
Integrating by parts we have
\begin{multline*}
I_1=\re\int_\TT\pax U(x\pm ikt)\frac{1}{2}\pax |\rho(x\pm ikt)|^2dx\\
-\re\; i\int_\TT\pax^2 U(x\pm ikt)\re\rho(x\pm ikt)\im\rho(x\pm ikt)dx\\
+ 2\int_\TT\pax \im U(x\pm ikt)\im\rho(x\pm ikt)\pax\re\rho(x\pm ikt)dx,
\end{multline*}
and we have
$$
I_1\leq c\|\rho\|_{L^2(S)}^2(\|\rho\|_{L^\infty(S)}+1)+\|\pax\im U\|_{L^\infty(S)}\|\rho\|_{L^2(S)}\|\pax \rho\|_{L^2(S)}.
$$
Due to the mean conservation and that the initial data has $\im \rho_0=0$ we have the bound
$$
\|\pax \im U\|_{L^\infty(S)}\leq c\|\pax^2\im U\|^{1/2}_{L^2(S)}\|\pax\im U\|_{L^2(S)}^{1/2}\leq c\|\im\rho\|_{L^2(S)},
$$ 
where in the last inequality we use the Poincar\'e inequality. 

The last term can be bounded as follows:
$$
\re\int_\TT\bar{\rho}(x\pm ikt)(\pm ik\pax \rho (x\pm ikt))dx\leq k\|\rho\|_{L^2(S)}\|\pax\rho\|_{L^2(S)}.
$$
Putting all together and using Sobolev embedding we have 
\begin{equation}\label{L2IA}
\frac{d}{dt}\|\rho\|_{L^2(S)}^2\leq c\|\rho\|_{H^3(S)}^2(1+\|\rho\|_{H^3(S)}).
\end{equation}
We have
$$
\frac{1}{2}\frac{d}{dt}\int_\TT|\pax^3\rho(x\pm ikt)|^2dx=\re\int_\TT\pax^3\bar{\rho}(x\pm ikt)(\pat \pax^3\rho(x\pm ikt)\pm ik\pax^4\rho(x\pm ikt))dx.
$$
Following the previous techniques we have 
\begin{multline*}
\frac{1}{2}\frac{d}{dt}\int_\TT|\pax^3\rho(x\pm ikt)|dx\leq c\|\rho\|_{H^3(S)}^2(\|\rho\|_{H^3(S)}+1)-\beta\int_\TT|\Lambda^{1/2}\pax^3\rho|^2dx\\
+\re\int_\TT\pax^3\bar{\rho}\pax^4\rho\pax Udx \pm k\re \;i\int_\TT\pax^3\bar{\rho}\pax^4\rho dx.
\end{multline*}
We write
$$
J_1=\re\int_\TT\pax^3\bar{\rho}\pax^4\rho\pax Udx\;\;\text{ and } J_2=\pm k\re \;i\int_\TT\pax^3\bar{\rho}\pax^4\rho dx.
$$
Splitting into real and imaginary part we have
\begin{multline*}
J_1=\re\int_\TT\pax U\frac{1}{2}\pax|\pax^3\rho|^2dx+\int_\TT\pax^3\re\rho\pax^4\im\rho\im\pax Udx-\int_\TT\pax^3\im\rho\pax^4\re\rho\im\pax Udx\\
\leq c\|\pax^3\rho\|_{L^2(S)}(\|\rho\|_{L^\infty(S)}+1)+2 K_1,
\end{multline*}
where the last integral is
\begin{multline*}
K_1=\bigg{|}\int_\TT\pax^3\re\rho\Lambda H\pax^3\im\rho\im\pax Udx\bigg{|}\\
=\bigg{|}\int_\TT\Lambda^{1/2} H\pax^3\im\rho\Lambda^{1/2}(\pax^3\re\rho\im\pax U)dx\bigg{|}\\
\leq\|\Lambda^{1/2} \pax^3\rho\|_{L^2(S)}\|\Lambda^{1/2}(\pax^3\re\rho\im\pax U)\|_{L^2(S)}.
\end{multline*}
We use the commutator estimate 
$$
\|\Lambda^{1/2}(fg)-f\Lambda^{1/2}g\|_{L^2(S)}\leq c\|\pax f\|_{L^\infty(S)}\|g\|_{L^2(S)}
$$
to conclude that
\begin{multline*}
K_1\leq\|\Lambda^{1/2} \pax^3\rho\|_{L^2(S)}(\|\pax^2\im U\|_{L^\infty(S)}\|\pax^3\rho\|_{L^2(S)}+\|\pax\im U\|_{L^\infty(S)}\|\Lambda^{1/2}\pax^3\rho\|_{L^2(S)})\\
\leq c(\|\Lambda^{1/2} \pax^3\rho\|_{L^2(S)}\|\rho\|_{L^\infty(S)}\|\pax^3\rho\|_{L^2(S)}+\|\im\rho\|_{L^2(S)}\|\Lambda^{1/2} \pax^3\rho\|_{L^2(S)}^2). 
\end{multline*}
If we use Young inequality we have
$$
K_1\leq \left(\frac{\epsilon}{2}+c\|\im\rho\|_{L^2(S)}\right)\|\Lambda^{1/2} \pax^3\rho\|_{L^2(S)}^2+\frac{1}{2\epsilon}\|\rho\|^2_{L^\infty(S)}\|\pax^3\rho\|^2_{L^2(S)}.
$$
We can also bound
$$
J_2\leq k\|\Lambda^{1/2}\pax^3\rho\|_{L^2(S)}.
$$
Finally, taking $\epsilon=\beta/2$
\begin{multline}\label{H3IA}
\frac{d}{dt}\int_\TT|\pax^3\rho(x\pm ikt)|dx\leq c\|\rho\|_{H^3(S)}^2(\|\rho\|_{H^3(S)}+1)\\
+\left(c\|\im\rho\|_{L^2(S)}+k-\frac{3\beta}{4}\right)\|\Lambda^{1/2}\pax^3\rho\|_{L^2(S)},
\end{multline}
and this is a correct bound if the term
$$
K_2(t)=c\|\im\rho\|_{L^2(S)}+k-\frac{3\beta}{4}\leq0.
$$
Initially $K_2(0)=k-\frac{3\beta}{4},$ so taking $k=\beta/4<\beta$ the condition $K_2<0$ will be satisfied for a (maybe short) time if we can bound the evolution of $\|\im\rho\|_{L^2(S)}$. 

We define the \emph{`energy'}
$$
\|\rho\|_{A(S)}=\|\rho\|_{H^3(S)}+\frac{1}{\frac{\beta}{2}-c\|\im\rho\|_{L^2(S)}}.
$$
The evolution of this energy is
$$
\frac{d}{dt}\|\rho\|_{A(S)}\leq c((\|\rho\|_{A(S)}+1)^2+\|\rho\|_{A(S)}^2\frac{d}{dt}\|\im\rho\|_{L^2(S)}).
$$
At this point it is easy to obtain
$$
\frac{d}{dt}\|\im\rho\|_{L^2(S)}\leq c(\|\rho\|_{H^3(S)}+1)\|\im\rho\|_{L^2(S)}^2,
$$
so 
$$
\frac{d}{dt}\|\rho\|_{A(S)}\leq (\|\rho\|_{A(S)}+1)^5
$$
and we conclude that there is a time $\tau(\rho)$ so that $\|\rho\|_{A(S)}<\infty.$

Now we approximate the problem in the real line as in the proof of the Theorem \ref{WellP}, using the rescaled heat kernel as a mollifier:
$$
\pat \rho^\epsilon=-\beta\jeps\Lambda^\alpha\jeps \rho^\epsilon+\jeps(\jeps\rho^\epsilon(\jeps\rho^\epsilon-1))+\jeps(\nabla\jeps\rho^\epsilon\cdot\nabla\jeps T(\rho^\epsilon)),
$$
where
$$\Delta T(\rho^\epsilon)= (\jeps\rho^\epsilon-1),$$
and we use $\rho_0^\epsilon=\jeps\rho_0$ as initial data.
Using Picard's Theorem in $H^3(\TT)$ we obtain solutions $\rho^\epsilon$ up to time $\tau^\epsilon$ which are analytic. Due to the above estimates we have a time a common time of existence $\tau$ depending on the initial data $\rho_0$ but not in $\epsilon$ and if $t<\tau(\rho_0)$, $\|\rho^\epsilon\|_{H^3(S)}<\infty$. Now we pass to the limit in the Hardy-Sobolev space $H^3(S)$ obtaining $\rho$ and we use the fact that if the energy $\|\cdot\|_{A(S)}$ remains bounded then this implies the analyticity. Due to the uniqueness, this is the same $\rho$ as in Theorem \ref{WellP}.

In the proof for the case $\alpha>1$ we use the inequality
$$
\|\Lambda^{1/2}T\|^2_{L^2(S)}\leq \|T\|_{L^2(S)}^2+\|\Lambda^{\alpha/2}T\|_{L^2(S)}^2
$$
obtained by splitting in the Fourier side.
We conclude the proof in the supercritical case following the same steps.
\end{proof}


Using the smoothing effect, one can prove ill-posedness in Sobolev spaces when the problem is considered backward in time.
We remark that this result does not require global existence of solutions. 

\begin{prop}[Ill-posedness]
There are solutions $\tilde{\rho}$ to the backward in time equation \eqref{eq7} with $\beta>0$, such that $\|\tilde{\rho}\|_{H^k}(0)<\epsilon$ and $\|\tilde{\rho}\|_{H^k}(\delta)=\infty$ for all $0<\epsilon$, sufficiently small $0<\delta$, and $k>3$.
\end{prop}
\begin{proof}
Let us fix $k=4$, the other cases being similar.
Take $g_0(x)\in H^ 3(\TT)$ but $g_0 \notin H^4(\TT)$ and consider the solution (forward in time) $\rho^{\lambda}$ to the equation \eqref{eq7} with initial data $\rho(x,0)=\lambda g_0(x)$ where $0<\lambda<1$. We have (see Theorem \ref{SE}) that there is a uniform (in $\lambda$) time $\delta^ *(g_0)$ such that $\rho^ \lambda$ exists and is analytic up to time $\delta^* (g_0)$. Now define $\tilde{\rho}^ {\lambda,\delta}(x,t)=\rho^ \lambda(x,-t+\delta)$ for fixed $0<\delta<\delta^*(g_0)$. We have that $\|\tilde{\rho}^ {\lambda,\delta}\|_{H^4}(\delta)=\lambda\|g_0\|_{H^4}=\infty.$ Now, using Theorem \ref{SE} we have that there is a growing strip of analitycity for $\rho^ \lambda$ where we can apply Cauchy's integral formula in order to obtain
$$
\|\pax^ 4 \tilde{\rho}^{\lambda,\delta}\|_{L^2(\TT)}(0)=\|\pax^ 4 \rho^ \lambda\|_{L^2(\TT)}(\delta)\leq \frac{C(\beta)}{\delta^*(g_0)}\|\pax^ 3 \rho^ \lambda\|_{L^2(S_{\delta^*})}\leq \frac{c(\beta,g_0)\lambda}{\delta}.
$$
Taking $0<\lambda<\min\{1,\epsilon\delta/c(\beta,g_0)\}$ we conclude the proof.
\end{proof}

\section{Gravitational stability}
\label{sec_InfNorm}

The blow-up of the solutions in finite time is of particular interest from the physical point of view, since it determines whether a gravitationally bound structure may form or not.
In order to investigate the necessary conditions for star formation, the present section is devoted to the evolution of $\rhomax(t)$.
More precisely, it is our goal to determine whether/when the growth of singularities can be avoided in our non-local adhesion model.
We will focus on the critical case ($\alpha=1$) and remark that the resulting equation is more dissipative if $\alpha>1$.
All our results are expressed in units such as $\rh=1$; the conversion to other systems can be trivially carried out by substituting every occurrence of $\rho$ by $\frac{\rho}{\rh}$ or, alternatively, $\beta$ by $\beta\rh$.

\subsection{Minimum time for star formation}

First of all, let us address the minimum time required for the gravitational collapse of a protostar.
For those cases where a star is formed (i.e. $\rho(x_*,t_*)\to\infty$, see e.g. Figure~\ref{blo}), we provide a lower bound for the value of $t_*$ as a function of the initial condition.
Note that, by virtue of Proposition \ref{SingF}, classical solution is guaranteed up to a time $t_*$.

\begin{prop}[Minimum time for blow-up]\label{StarF}
Let $\rho\geq0$ be a smooth solution of \eqref{eq7} in $\TT^d$, $d\leq3$. Then we have
$$
\frac{d}{dt}\rhomax(t)\leq \rhomax(t)(\rhomax(t)-1), 
$$
thus
$$
\rhomax(t)\leq \frac{\rhomm}{\rhomm+(1-\rhomm)e^{t}}.
$$
As a consequence the minimum time for the formation of a star is
$$
t_* = \log\left(\frac{\rhomm}{\rhomm-1}\right).
$$
\end{prop}
\begin{proof}
Due to the smoothness of $\rho$ and using Rademacher Theorem, we have that the evolution of
$$
\rhomax(t)=\max_{x} \rho(x,t)=\rho(x_t),
$$ 
is given by
$$
\frac{d}{dt}\rho(x_t)\leq -\beta\Lambda^\alpha \rho(x_t)+\rho(x_t)(\rho(x_t)-1)\leq \rho(x_t)(\rho(x_t)-1),
$$
where the last inequality is due to the integral representation for the diffusive operator. Now the result is straightforward.
\end{proof}

\subsection{Exponential suppression of density fluctuations}

Our main result, and the one that bears more physical relevance, is that equation \eqref{eq7} with $\alpha=1$ does not have a blow-up for $\rhomax(t)$ -- i.e. no star can ever form -- if $\beta$ is big enough.
Physically, $\beta$ is the dimensionless ratio of the free-fall time, representative of the strength of the gravitational forces, and the sound-crossing time, representative of the restoring pressure force.
Depending on the amplitude of the initial conditions (more precisely, the value of \rhomm), there is a critical value of the sound speed above which the pressure gradient is able to overcome gravity.
Or, in other words, for a uniform gas cloud at a given temperature (i.e. $\beta$), sufficiently small perturbations cannot grow.
Mathematically, this result is a Maximum Principle for $\rhomax$ for a special class of initial data, and the description in terms of $\beta$ or $\Vert\rho_0\Vert_{L^1}$ merely reflects the choice of units.
Both choices are, of course, absolutely equivalent.

Consistently with the results of numerical simulations, we find that, for sufficiently small perturbations and/or large pressure forces, there is a depletion of the solution (see Figure~\ref{dec}), and the system evolves towards the homogeneous equilibrium state $\rho(x)=1$.
Since $\rhomax$ is bounded, global existence of classical solution follows from Proposition~\ref{SingF}.

\begin{figure}[ht]
\includegraphics[scale=0.45]{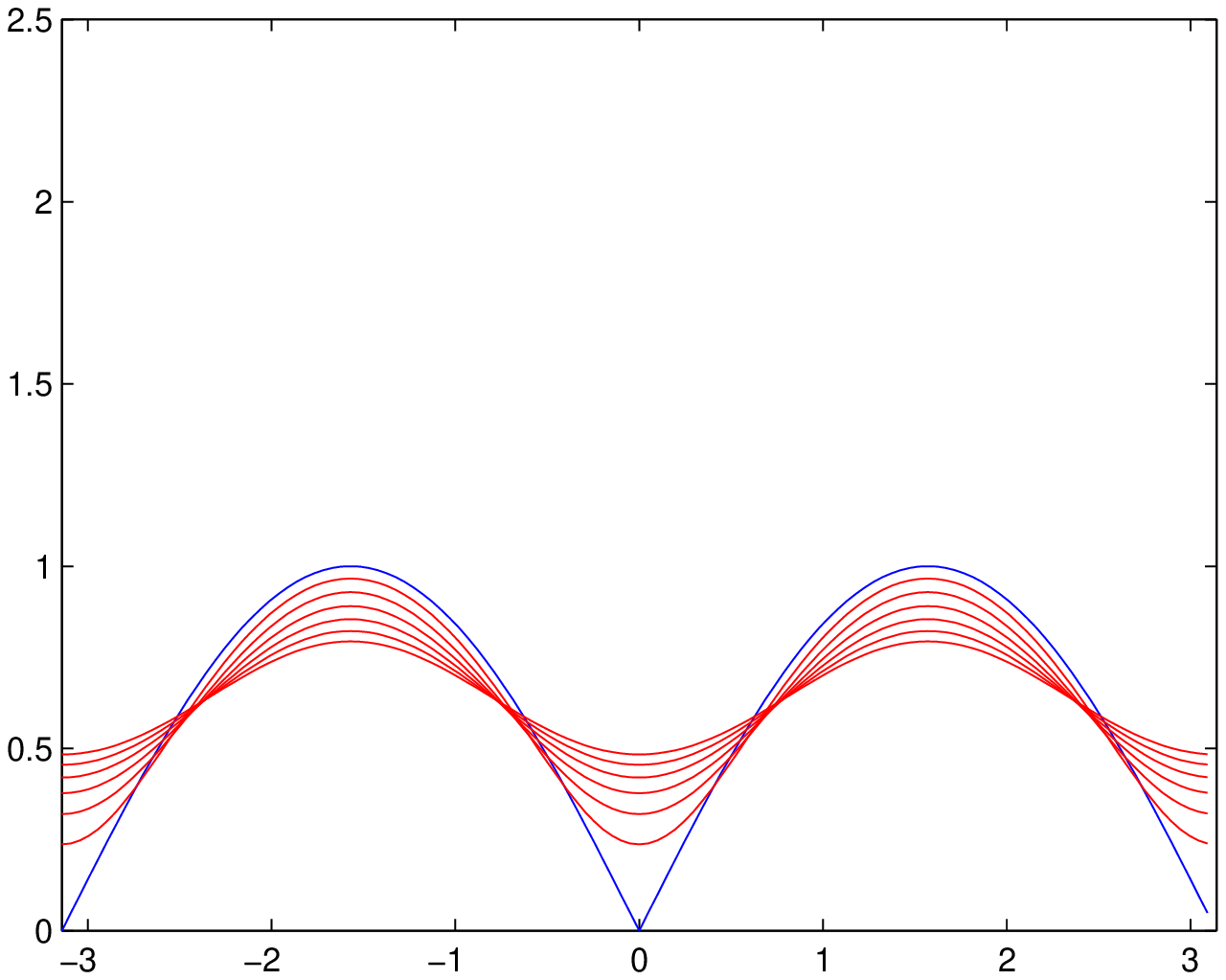} 
\includegraphics[scale=0.45]{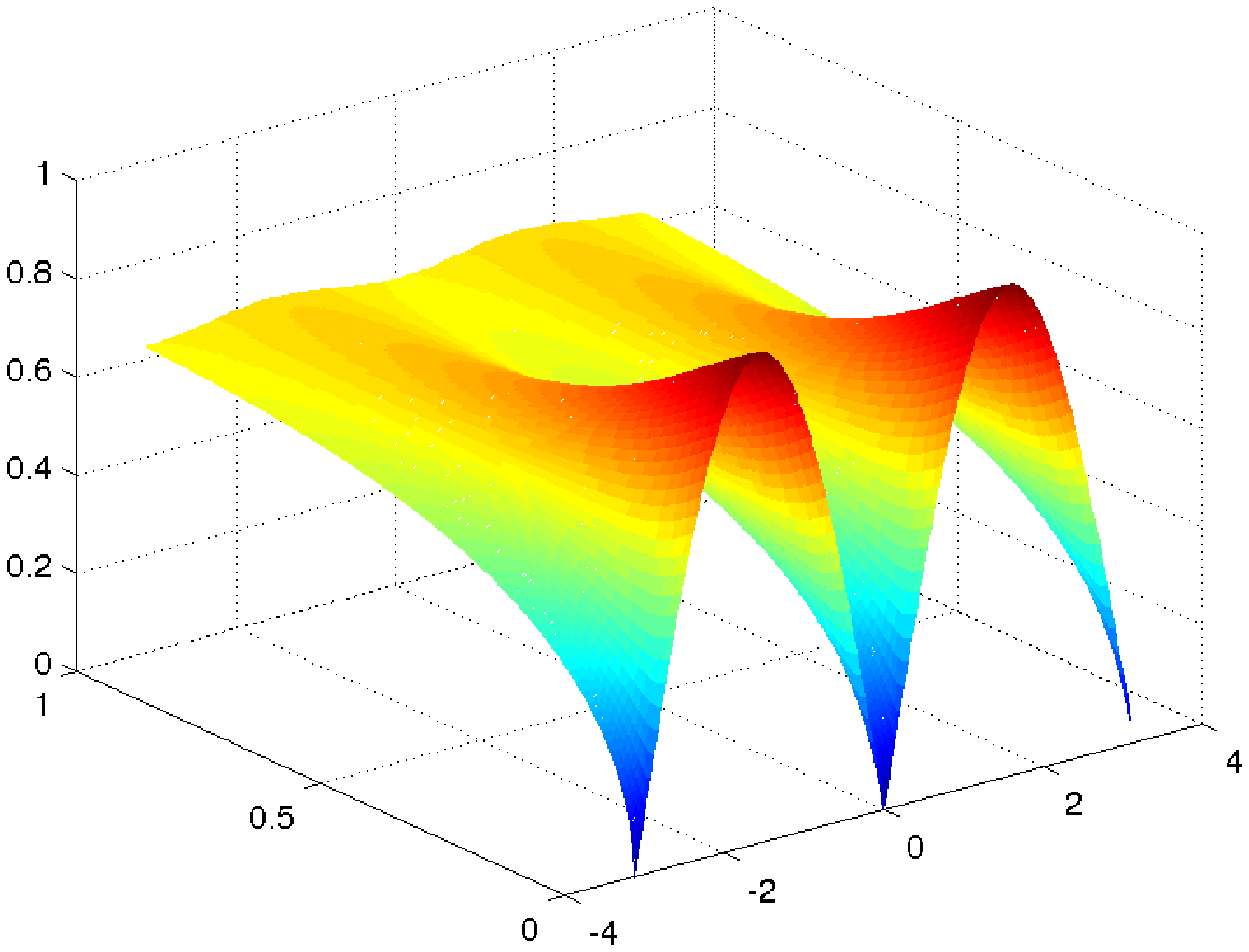}
\caption{a) The initial data (in blue) and the evolution (in red) for the one-dimensional case of equation \eqref{eq7} with $\beta=1$ and b) the complete evolution.}
\label{dec}
\end{figure}

\begin{figure}[ht]
\includegraphics[scale=0.45]{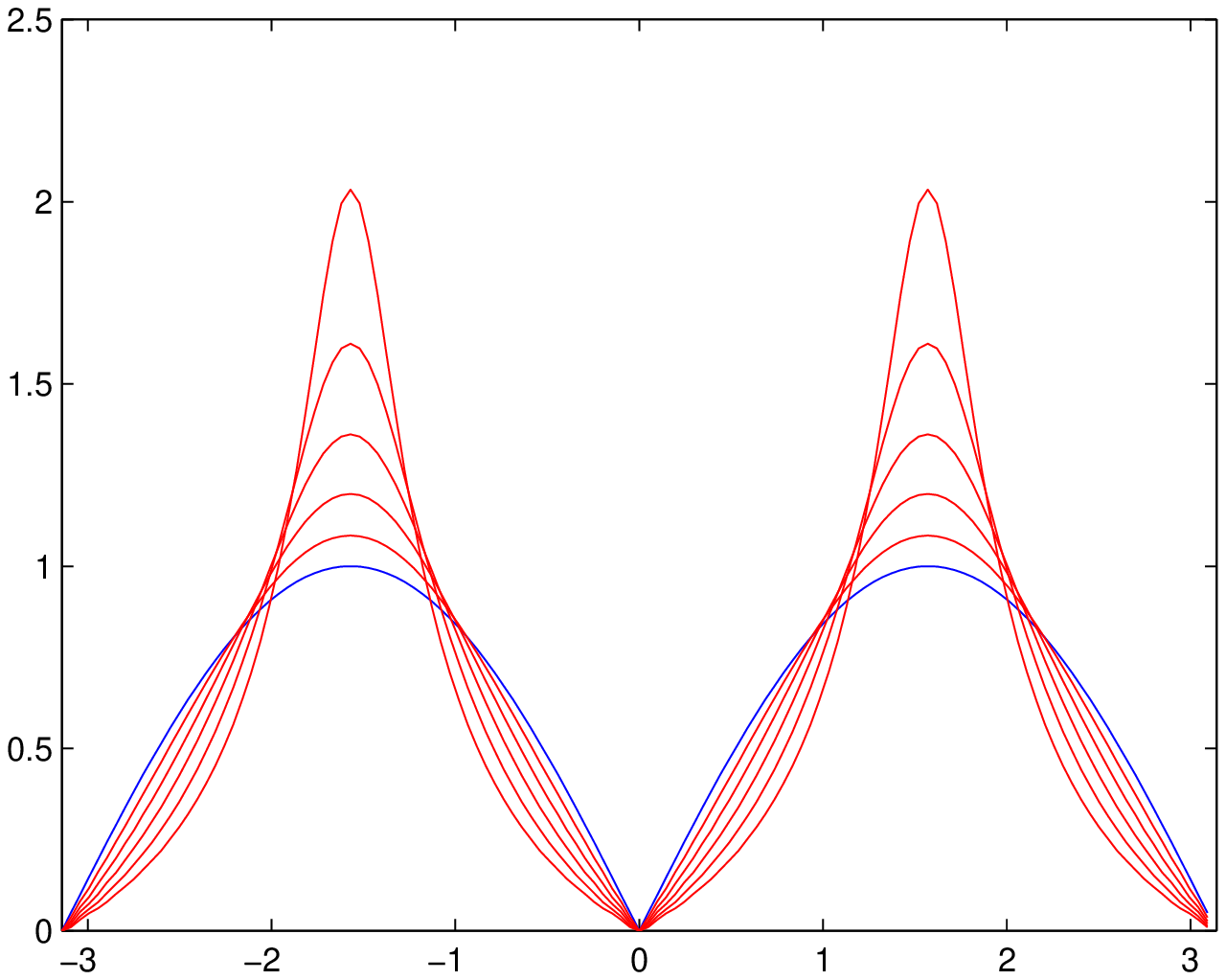}
\includegraphics[scale=0.45]{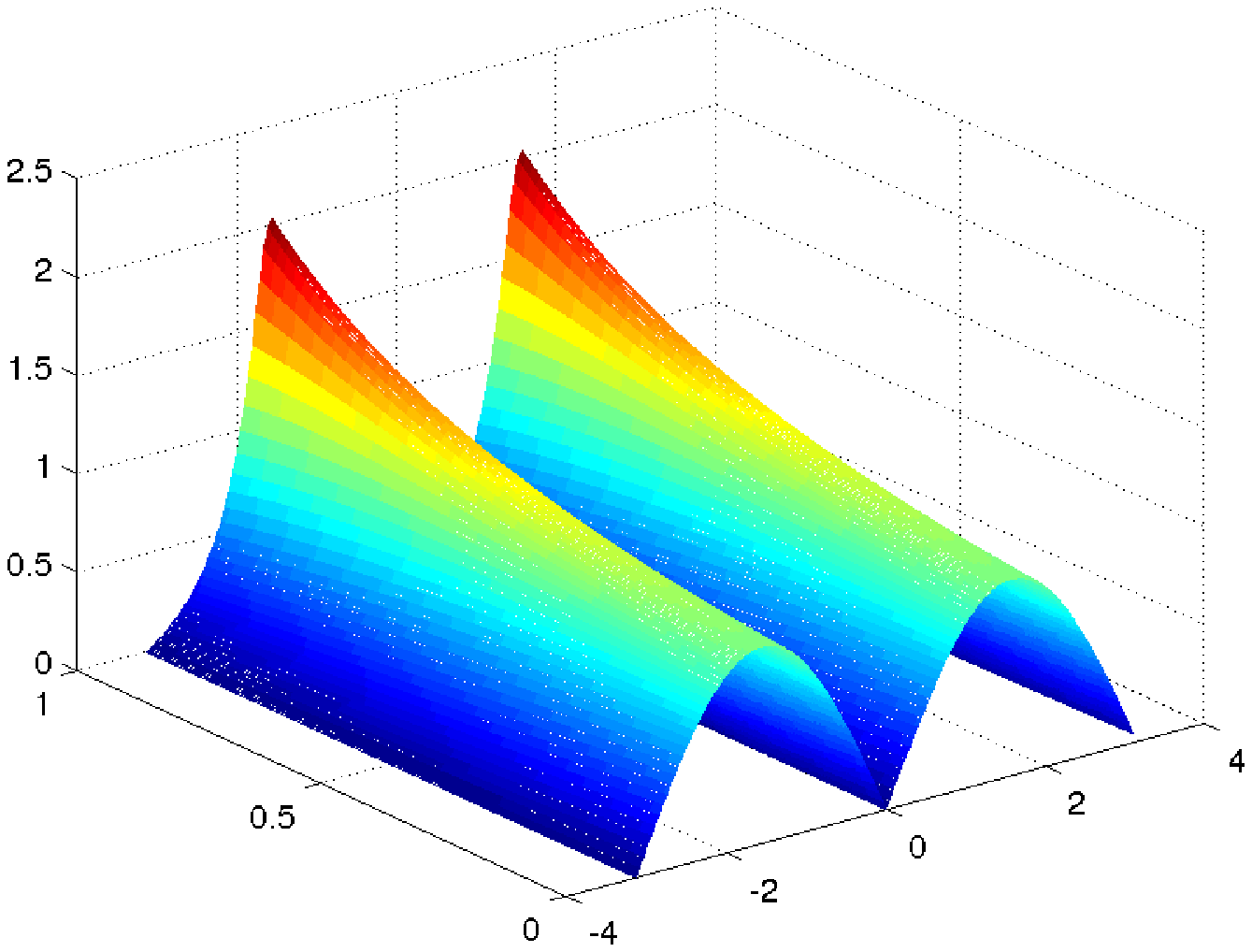} 	
\caption{a)The initial data (in blue) and the evolution (in red) for the one-dimensional case of equation \eqref{eq7} with $\beta=0$ and b) the total evolution.}
\label{blo}
\end{figure}

\begin{teo}[Decay of small perturbations]
\label{Max1_d2}
For \eqref{eq7} with $0<\alpha<2$, spatial domain $\TT^d$, $d=1,2,3$, and initial data $\rho_0$, if
$$\rhomm \leq c_{\alpha,d}\beta,$$
where $c_{\alpha,d}$ is an explicit constant, then the maximum density decays exponentially,
$$
\rhomax(t) ~\leq~ 1 \,+\, \left[\, \rhomm -1 \,\right]\ e^{- \left[\,c_{\alpha,d}\beta - \rhomm \right]\, t}. 
$$
Moreover, we have that $c_{1,1}=1$.
\end{teo}

\begin{proof}
First, we prove the case $\alpha=1,$ $d=1$. Denoting the point where $\rho$ achieves its maximum value by $x_t$ and using the integral representation of the fractional Laplacian, one can easily derive the bound
$$
\Lambda\rho(x_t)
=
\frac{1}{2\pi}\text{P.V.}\int_\TT\frac{\rho(x_t)-\rho(y)}{\sin^{2}\left(\frac{x_t-y}{2}\right)}dy
~\geq~
\frac{1}{2\pi} \int_\TT \rho(x_t)-\rho(y)\, dy = \rho(x_t)-1
$$
Thus
$$
\pat \rho(x_t)
\,\leq\,
-\beta \left[\, \rho(x_t) -1 \,\right] +  \rho(x_t) \left[\, \rho(x_t) -1 \,\right]
\,=\,
\left[\, \rho(x_t) -1 \,\right] \left[\, \rho(x_t) - \beta \,\right],
$$
and the maximum density will always decrease if $\rho(x_t) < \beta$.
In fact,
$$
\pat \left[\, \rho(x_t) -1 \,\right]
\,\leq\,
- \left[\, \beta - \rhomm \right] \left[\, \rho(x_t) -1 \,\right].
$$
Now we proceed with the $d-$dimensional case.
From the integral representation of the fractional Laplacian in $\TT^d$
$$
\Lambda^\alpha\rho(x_t)=C_{\alpha,d}\sum_{\nu\in\mathbb{Z}^d}\text{P.V.}\int_{\TT^d}\frac{\rho(x_t)-\rho(y)}{|x_t-y-2\pi\nu|^{d+\alpha}}dy
$$
we have that
$$
\Lambda^\alpha\rho\geq C_{\alpha,d}\ \text{P.V.}\int_{\TT^d}\frac{\rho(x_t)-\rho(y)}{|x_t-y|^{d+\alpha}}dy.
$$
Since $|x_t-y| \leq \pi \sqrt{d}$,
$$
\Lambda^\alpha\rho \geq c_{\alpha,d} \left[ \rho(x_t)-1 \right].
$$
and we can conclude the result analogously to the one-dimensional case.
\end{proof}

In one spatial dimension, it is also possible to prove that, for sufficiently large $\beta$ (or, equivalently, sufficiently small $\Vert\rho_0\Vert_{L^1}$), \emph{all} density fluctuations are exponentially suppressed, regardless of their initial amplitude.
\begin{teo}[Global stability for $d=1$]
\label{Max}
Given equation \eqref{eq7} with $\alpha=1$, spatial domain $\TT$, and initial data $\rho_0$, if $$\beta \geq 4\pi^2,$$ then the maximum density decays exponentially,
$$
\rhomax(t) ~<~ 1 \,+\, \left[\, \rhomm -1 \,\right]\ e^{-t}.
$$
Moreover, if $1<\alpha<2$ then, for any $\beta>0$, the following inequality holds 
$$
\|\rho\|_{L^\infty}\leq c(\alpha,\beta).
$$
\end{teo}

\begin{proof}
Let us consider the neighbourhood around the absolute density maximum
$$
\Omega_\delta=\{\, y ~|~ |x_t-y|\leq2\delta\, \},
$$
where $\delta$ is a parameter that will be set later, and decompose it in two subsets, $\Omega_\delta = \mathcal{U}_1\cup\mathcal{U}_2$, such that
$$
\mathcal{U}_1=\{\, y\in\Omega_\delta ~|~ \rho(x_t)-\rho(y)\geq \rho(x_t)/2\, \},
$$
and
$$
\mathcal{U}_2=\{\, y\in\Omega_\delta ~|~ \rho(x_t)-\rho(y)<\rho(x_t)/2\, \}.
$$

Since $|x|\geq |\sin(x)|$,
$$
\Lambda \rho(x_t)
~\geq~
\frac{1}{2\pi}\int_{\Omega_\delta} \frac{\rho(x_t)-\rho(y)}{\left(\frac{x_t-y}{2}\right)^{2}}dy
~\geq~
\frac{1}{2\pi}\int_{\mathcal{U}_1} \frac{ \rho(x_t)/2 }{ \delta^2 }dy
~=~
\frac{1}{2\pi}\frac{\rho(x_t)}{2\delta^{2}} \left(|\Omega_\delta|-|\mathcal{U}_2|\right).
$$
Substituting $|\Omega_\delta| = 4\delta$ and using that
$$
2\pi = \int_\TT \rho(y)dy ~>~ \frac{\rho(x_t)}{2}|\mathcal{U}_2|,
$$
one obtains the lower bound
$$
\Lambda \rho(x_t) ~>~ \frac{1}{2\pi} \frac{\rho(x_t)}{2\delta^{2}} \left[ 4\delta - \frac{4\pi}{\rho(x_t)} \right],
$$
which is maximized for the choice
$$
\delta = \frac{ 2\pi }{ \rho(x_t) }.
$$
It is important to note, though, that the compactness of the domain also imposes that $\delta \leq \pi/2$ or, in other words,
$$
\rho(x_t)\geq 4.
$$

Inserting the bound for the dissipative term in equation~\eqref{eq7}, we have
$$
\pat \rho(x_t)
\leq
- \beta\,\frac{ \rho^2(x_t) }{ 4\pi^2 } + \rho(x_t) \left[\, \rho(x_t) -1 \,\right]
=
\rho^2(x_t) \left[ 1 - \frac{\beta}{4\pi^2} \right] - \rho(x_t)
$$
If $\beta \geq 4\pi^2$,
$$
\pat \rho(x_t) ~\leq~ - \rho(x_t) ~<~ - \left[\, \rho(x_t) -1 \,\right],
$$
which proves the desired result for $\rho(x_t) \geq 4$.
Theorem~\ref{Max1_d2} implies that this bound (indeed, a stricter one) also applies for $\rho(x_t) < 4$.
This concludes with the first part of the result.
In the case $1<\alpha<2$ we have
$$
\Lambda \rho(x_t) ~>~ \frac{1}{2\pi} \frac{\rho(x_t)}{2\delta^{1+\alpha}} \left[ 4\delta - \frac{4\pi}{\rho(x_t)} \right].
$$
Consequently, we get
$$
\pat \rho(x_t)
\leq
- \beta\,\frac{ \rho^{1+\alpha}(x_t) }{ (2\pi)^{1+\alpha} } + \rho(x_t) \left[\, \rho(x_t) -1 \,\right],
$$
and we conclude the proof.
\end{proof}

\begin{coment}
\label{th_L1norm}
Recall that in 2D \eqref{eq7.1} has a critical mass phenomenon, \emph{i.e.} if $\|\rho_0\|_{L^1(\RR^2)} < 8\pi$, then there is a global solution, while if $\|\rho_0\|_{L^1(\RR^2)} > 8\pi$ there is a finite-time blow-up. If we take $\beta=1$ and an initial datum with arbitrary $L^1$ norm, Theorem \ref{Max} gives us that, if $\|\rho_0\|_{L^1}\leq \frac{1}{2\pi}$ then the maximum decays exponentially and the solution is global.
\end{coment}

\cambio{
\begin{coment}
In order to show that the system tends to the homogeneous state\footnote{\cambio{At least, as far as the $L^\infty$ norm is concerned.}} for large $\beta$, let us consider the evolution of the minimum value $\rho(x_t)$.
Defining
$$ g(x) = 1-\rho(x), $$
$0 \le g(x_t) \le 1$ is maximum, and, as the function is smooth, its evolution is given by
$$
\pat g(x_t) = -\beta\Lambda^\alpha g(x_t) + \rho(x_t) g(x_t).
$$
Following the same reasoning as in Theorem~\ref{Max1_d2}, we get the bound 
$$
\Lambda^\alpha g(x_t) \geq c_{\alpha,d}\, g(x_t),
$$
and therefore
$$
\pat g(x_t) \leq [ \rho(x_t) - \beta c_{\alpha,d} ]\ g(x_t).
$$
Since $\rho(x_t)\leq 1$, the minimum will always increase towards $g(x_t) \to 0$ if $ \beta c_{\alpha,d} > 1 $.
Such condition is automatically fulfilled if $ \beta c_{\alpha,d} \geq \rhomm $ (Theorem~\ref{Max1_d2}) or $ \beta \geq 4\pi^2 $ for $\alpha=1$ in one spatial dimension, where $ c_{1,1} = 1 $ (Theorem~\ref{Max}).
\end{coment}
}

\section{Numerical simulations}
\label{sec_Sims}

In order to simulate the equation \eqref{eq7} we use a Fourier-collocation method (see \cite{can}). We use Fast Fourier Transform techniques to solve the Poisson equation and to approximate the spatial part of \eqref{eq7}. For the time integrator we use an explicit Runge-Kutta of order 4.

We center our attention in the 2D case. The other cases are analogous. We consider an equispaced discretization of our spatial domain $\TT^2$, 
$$\mathcal{T}^2_N = \left\{(x_k,y_l)\,:\, x_k = \frac{\pi}{N} \left(2k-N\right),\, y_l = \frac{\pi}{N}\left (2l-1\right),\, k,l = 0,\ldots, N-1\right\},$$
where $N$ is the number of points on each interval $[-\pi,\pi]$. We consider as our approximate solution the pair $\rho(x,y)$ and $T(\rho)(x,y)$ in the grid $\mathcal{T}^2_N$. We have
$$\rho(x_k,y_l) = \sum_{n,m=-N/2}^{N/2-1} \tilde{\rho}_{n,m} e^{i n\,x_k} e^{i m\,y_l},
\quad -(n^2 + m^2) \tilde{T(\rho)}_{n,m} = 4\pi G \tilde{\rho}_{n,m},\quad \tilde{T(\rho)}_{0,0} = 0,$$
and it is important to remark that the coefficients $\tilde{\rho}_{n,m}$ are related to the so-known Fourier coefficients, typically $\hat{\rho}_{n,m}$, but they are not the same (see \cite{can}). 
In the same way we can approximate all terms in \eqref{eq7.1}. Once we compute derivatives of $\rho$ and $U$ via the FFT in the equation, we get back to the physical space using the IFFT.

This ends the space discretization part. We consider our equation in the time interval $[0,\tau]$, for some $\tau>0$. As before, let $\mathcal{G}_J$ an equispaced partition of the time interval,
$$\mathcal{G}_J = \{t_s\,:\, t_s = s \tau/J,\, s = 0,\ldots, J-1\},$$
where $J$ is the number points in the time grid. For the evolution part of the equation, $\pat \rho$, we are going to use a Runge-Kutta method (specifically, the classic explicit method of order 4), and we start with the initial data
$$\rho_0(x,y) = |\sin x| + |\sin y|$$
represented in Figure~\ref{2d_0}.

\begin{figure}[ht]
\includegraphics[scale=0.45]{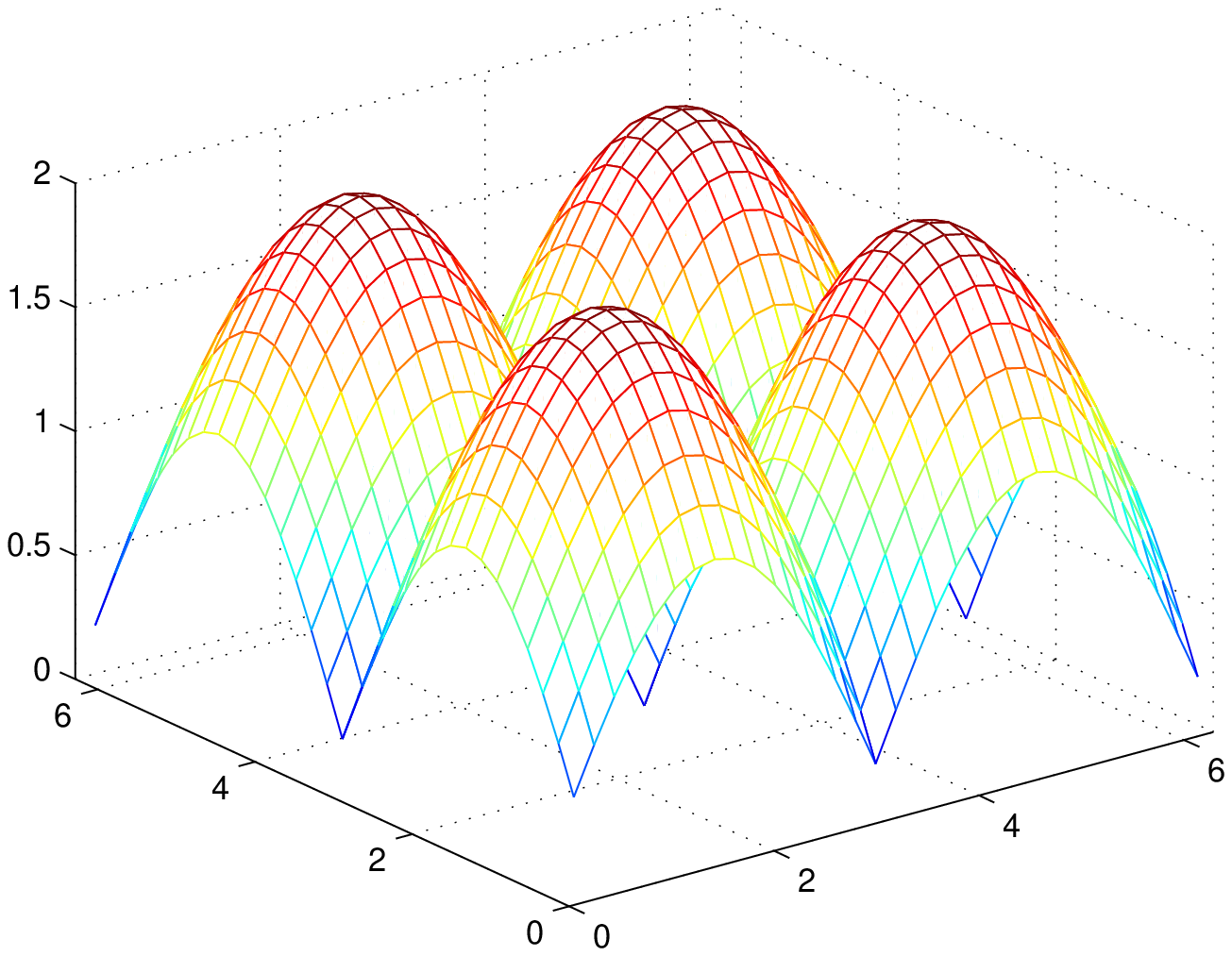} \includegraphics[scale=0.45]{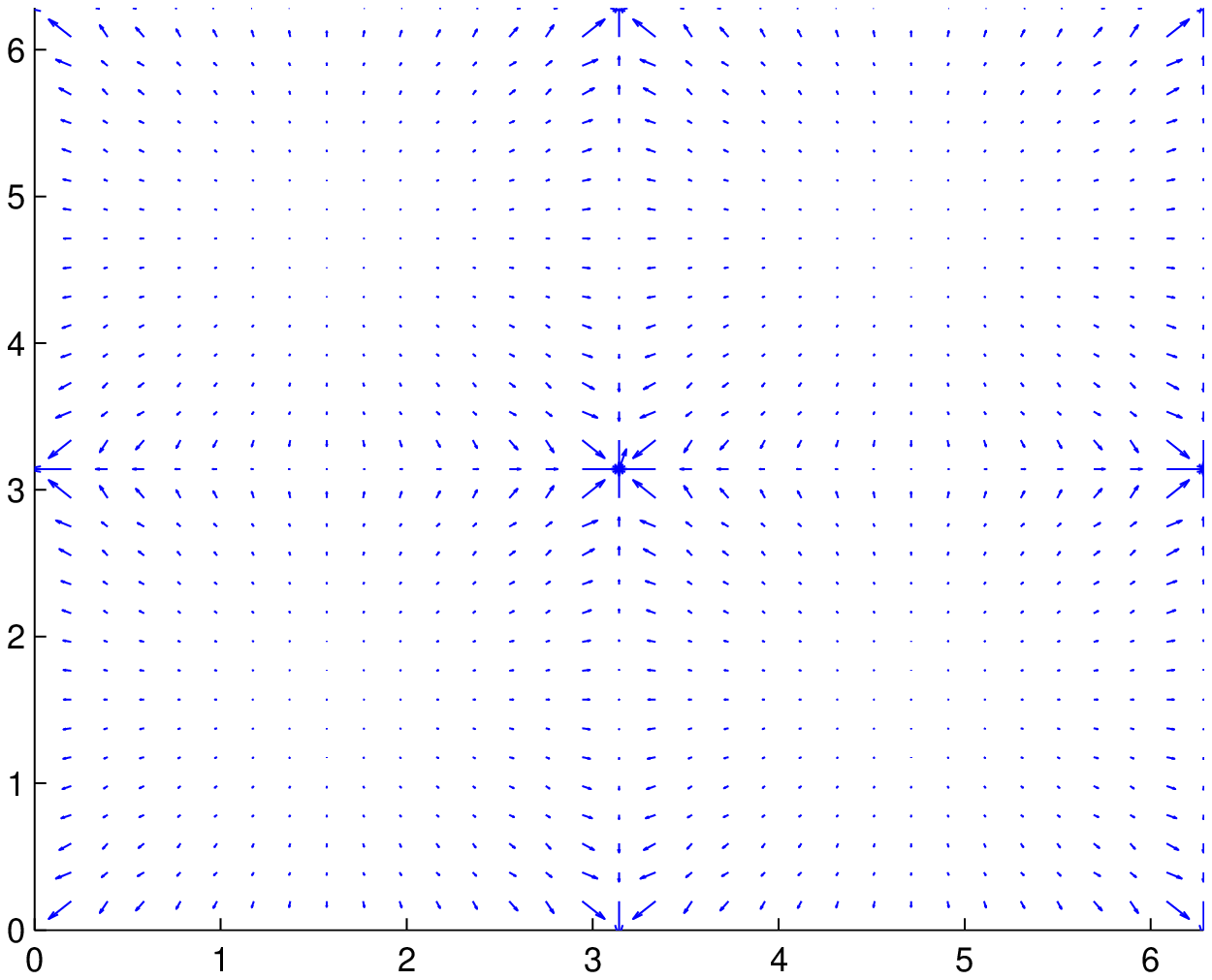} 	
\caption{Initial data for $\rho$ and $v$ with $\beta=10.$}
\label{2d_0}
\end{figure}
\begin{figure}[ht]
\includegraphics[scale=0.45]{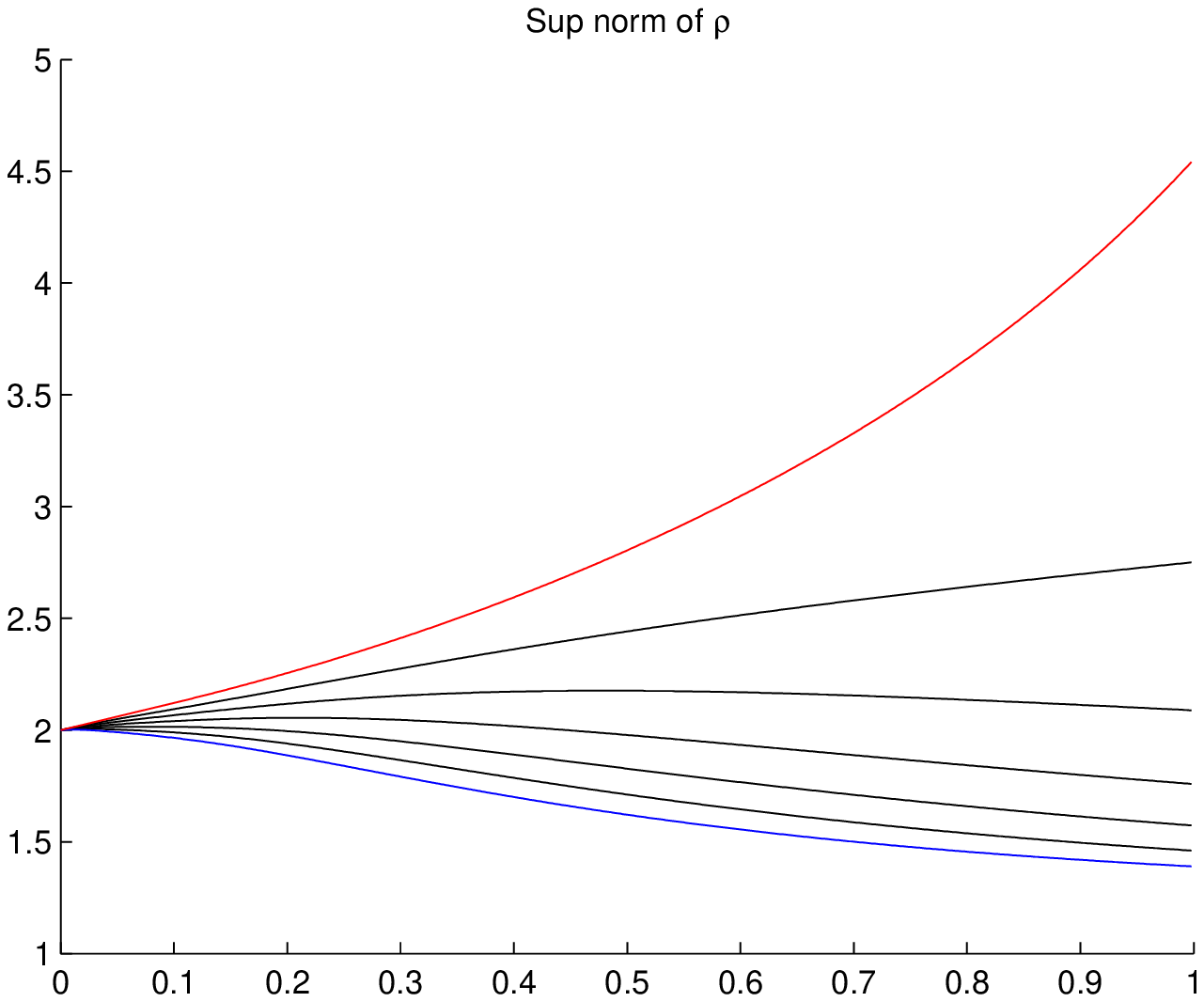} \includegraphics[scale=0.45]{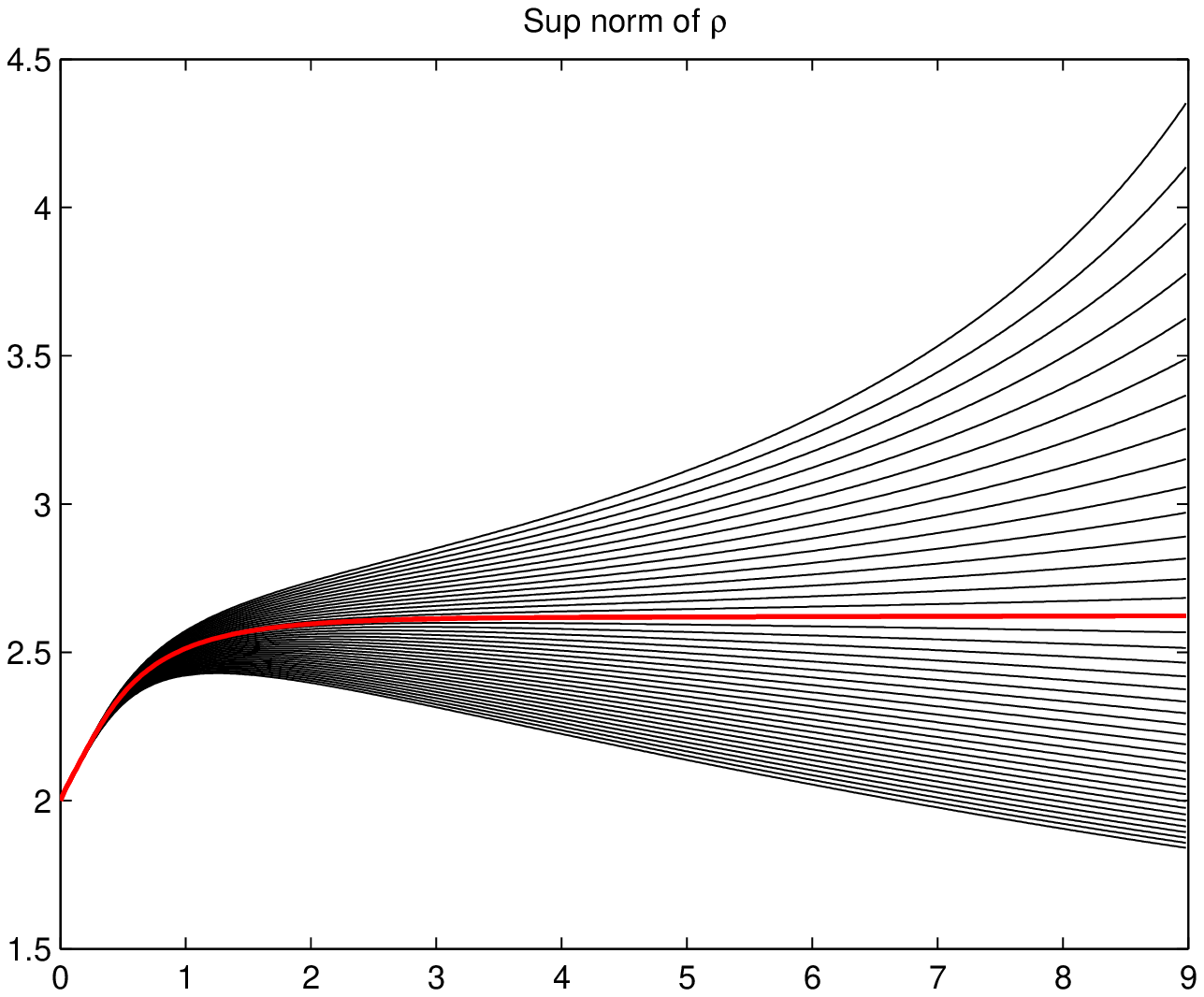} 	
\caption{a) The evolution of $\rhomax(t)$ for $\beta=0.2$ up to $0.8$, with final time $T=1$, and b) the evolution of $\rhomax(t)$ for $\beta=0.32$ up to $0.34$, with final time $T=9$.}
\label{2d_1}
\end{figure}

We studied the evolution of $\rhomax(t)$ for various $\beta\in [0,1]$.
On the left panel in Figure~\ref{2d_1}, we highlight in red the case $\beta=0.2$ and in blue $\beta=0.8$.
On the right panel, we plot the results for values of  $\beta$ in the interval $[0.32, 0.34]$.

The mechanism for singularity formation, i.e. the blow-up for $\rhomax(t)$, is illutrated in Figure~\ref{2d_blow}.
In Figure~\ref{2d_3}, one can see how the solution is depleted for large $\beta$.

\begin{figure}[ht]
\includegraphics[scale=0.45]{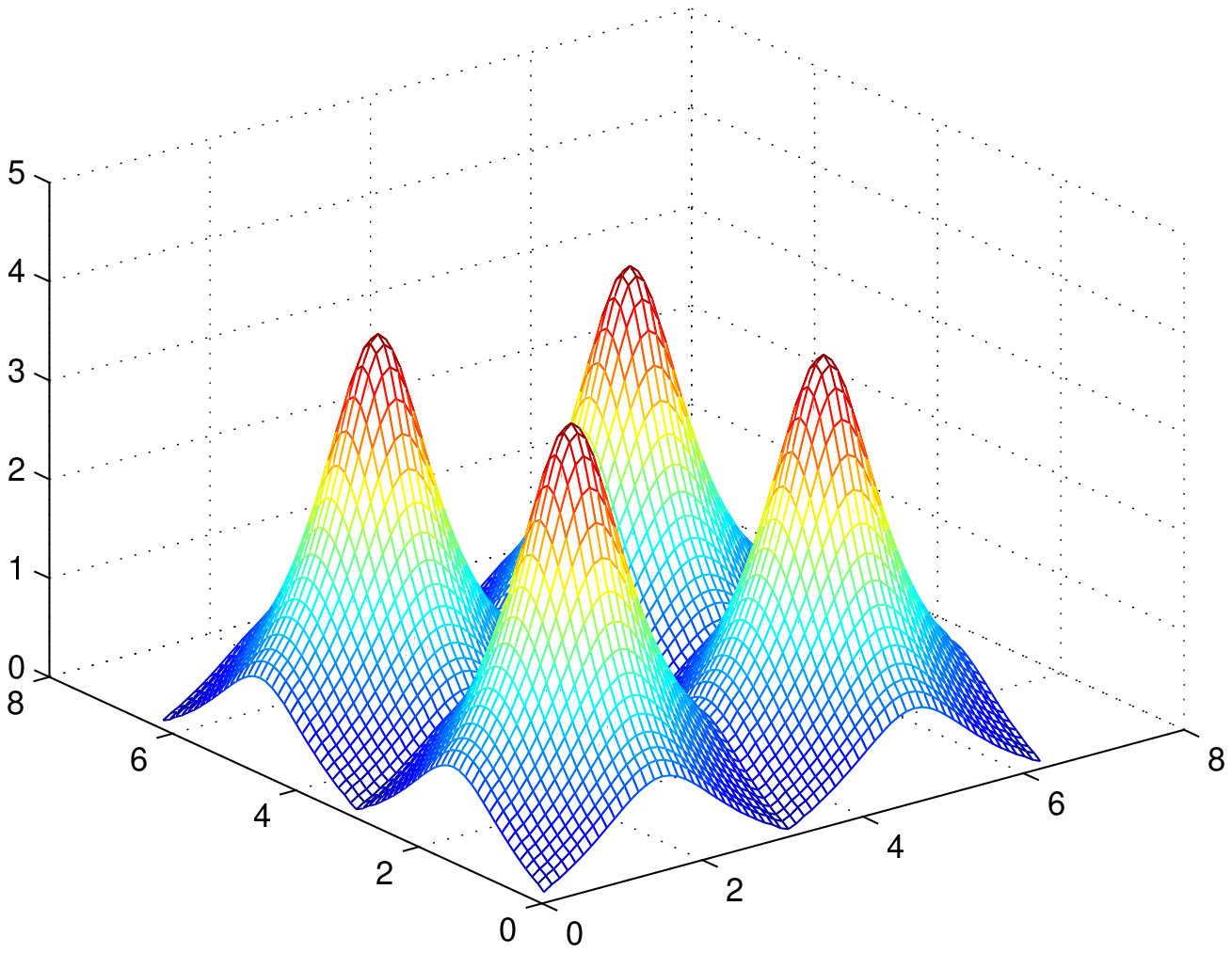} \includegraphics[scale=0.45]{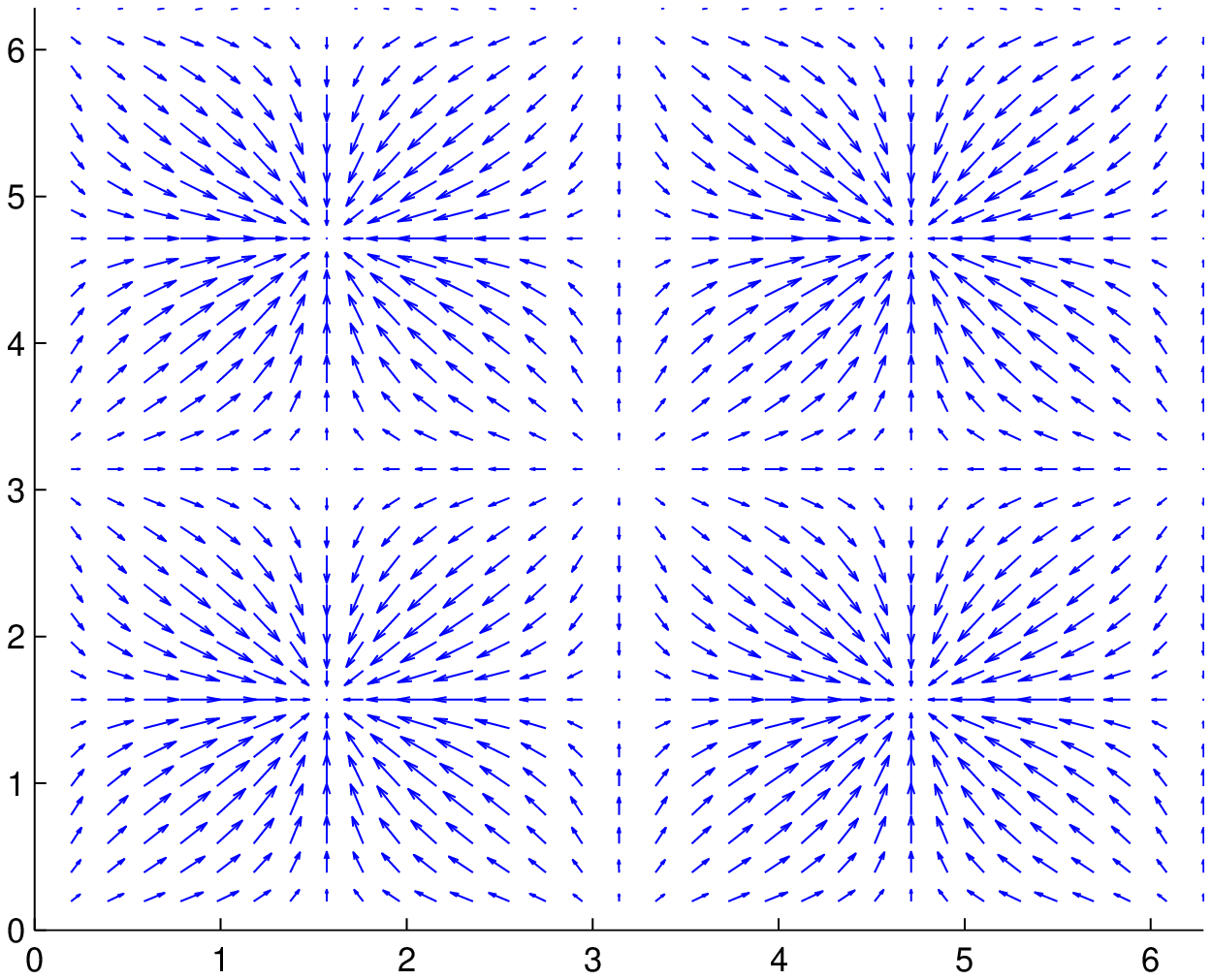} 	
\caption{a) $\rho$ for time $t=0.5$ and $\beta=0$ and b) $v$ at the same time.}
\label{2d_blow}
\end{figure}
\begin{figure}[ht]
\includegraphics[scale=0.45]{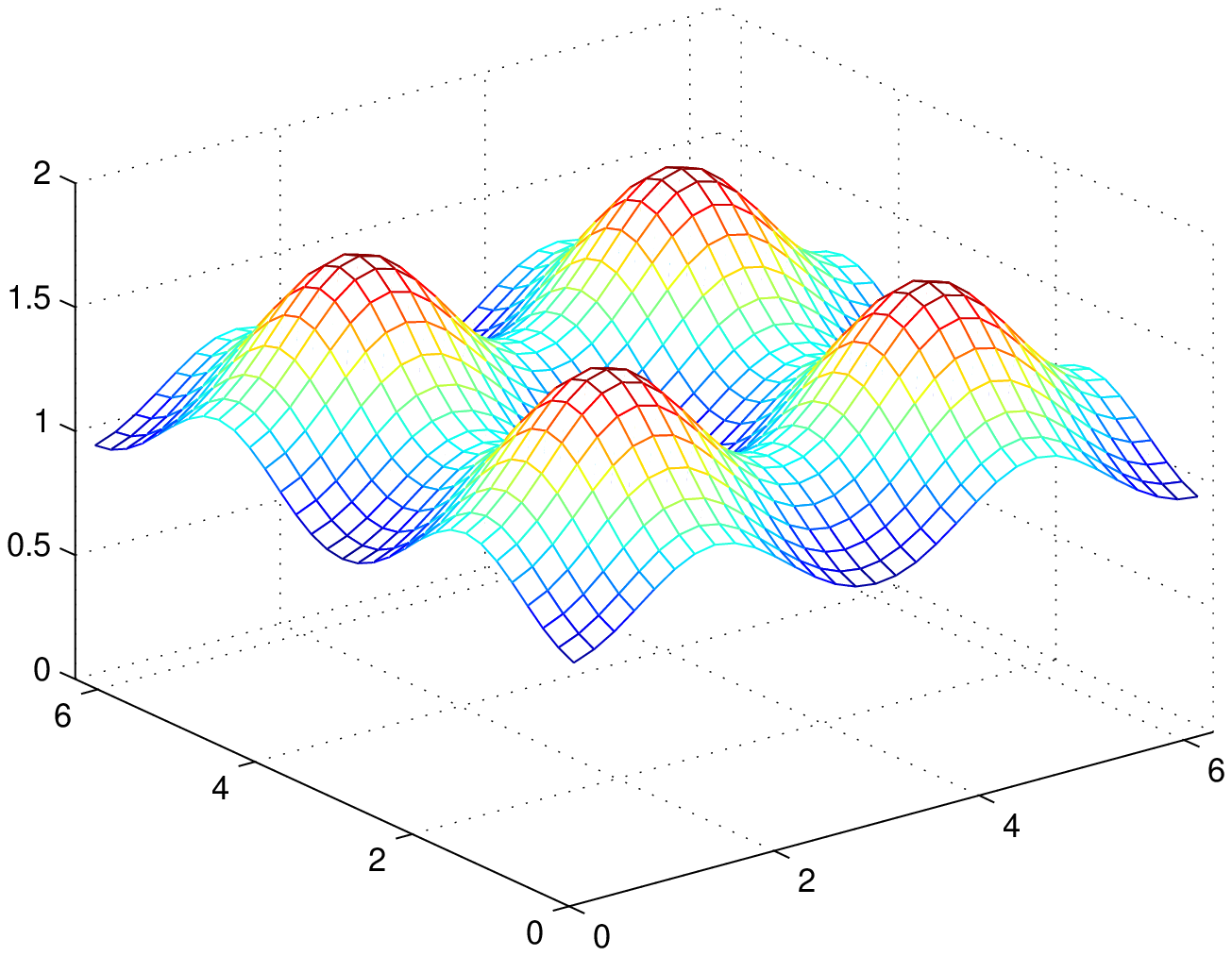} \includegraphics[scale=0.45]{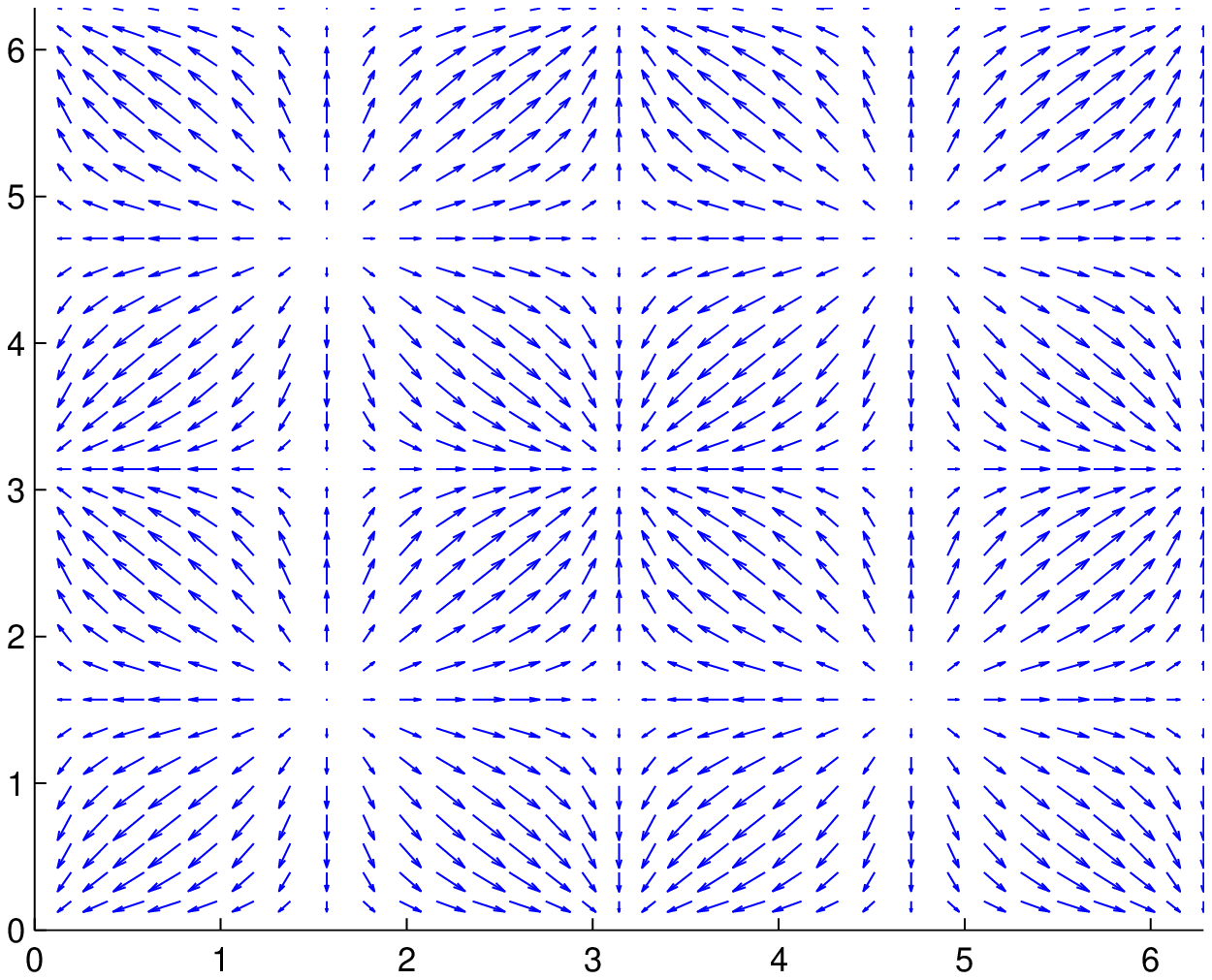} 	
\caption{a) $\rho$ for time $t=1$ and $\beta=0.5$ and b) $v$ at the same time.}
\label{2d_3}
\end{figure}

\section*{Acknowledgments}
We would like to thank the anonymous referees for a critic, yet extremely constructive report, that has improved the clarity of the presentation, and has helped us to put our results into context.
Y.~Ascasibar receives financial support from project AYA2010-21887-C04-03 from the former \emph{Ministerio de Ciencia e Innovaci\'on} (MICINN, Spain), as well as the \emph{Ram\'{o}n y Cajal} programme (RyC-2011-09461), now managed by the \emph{Ministerio de Econom\'{i}a y Competitividad} (fiercely cutting back on the Spanish scientific infrastructure).
R.~Granero is supported by grants MTM2008-03754, MTM2011-26696 and SEV-2011-0087 of the MICINN.

\appendix

\section{The `Jeans swindle'}
\label{sec_Jeans}

It is fairly common to see the Poisson equation expressed as
$$ \Delta U = \rho $$
without including the average density term that appears in equation~\eqref{eq_Poisson}.
Both expressions are equivalent for any finite mass (or, in general, charge) distribution embedded in an infinite space, because such configurations fulfil by definition $\rh=0$.
There are, however, important differences when this is not the case, and we would like to advocate the use of equation~\eqref{eq_Poisson} in order to describe an infinite system.
The convenience the \rh\ term is very clearly illustrated by considering the Fourier transform of equation~\eqref{eq_Poisson}
$$ -k^2\, \hat U_k = \hat \rho_k - \rh\delta_{k0} $$
for $k=0$.
If $\hat\rho_0 \equiv \rh \neq 0$ and the extra term was not included, the Poisson equation could never be solved for a periodic system.

The pioneering analysis \cite{Jeans_1902} by Sir James H. Jeans at the beginning of the past century was the first attempt to address the stability of an initially uniform gas cloud at a given temperature $T$ and density $\rho_0$ -- \emph{assumed} to be in equilibrium -- against the growth of arbitrarily small perturbations.
The assumption of initial equilibrium is inconsistent with the Poisson equation when the \rh\ term is neglected\footnote{Otherwise, it is easy to verify that a solution of the form $\rho(\vv{r},t) = \rho_0\, a(t)$ can be found, where the cosmic scale factor $a(t)$ is exactly the same as in a fully relativistic Lemaître-Friedmann-Robertson-Walker universe.
As shown in~\cite{A04}, the analogy between Newtonian and relativistic dynamics can be trivially extended to a spherically-symmetric Tolman-Bondi universe with cosmological constant.
The effect of the \rh\ term is indeed so similar to a constant `vacuum energy density' that it is tempting to associate them.
Some claims have already been made in this direction \cite{CalderLahav08}, and many studies attempt to accommodate current observational data on the accelerated expansion of the universe in terms of inhomogeneous models without a dark energy component (see e.g \cite{Buchert08} for a recent review).}, and is often referred to in the Astrophysical literature as the `Jeans swindle' (see e.g. \cite{BT}).
Nevertheless, several \cambio{recent works \cite{Ki03,Joy08,Ersh11,Chf,Fal13}} vindicate its validity using a variety of arguments, and we would like to claim that the reasoning above supports their main conclusions.
In particular, we claim that the law of gravity should be such that an infinite fluid with constant density is an (unstable) static equilibrium solution, where the net force felt by every point, averaged over such a homogeneous and isotropic background, vanishes.


By linearizing the hydrodynamic equations and decomposing the perturbations in Fourier modes, one finds the dispersion relation
\begin{equation}
\label{eq_Jeans}
 \omega^2 = k^2\cs - S_d G \rho_0
\end{equation}
where $\omega$ denotes the angular frequency of the oscillation and $k\equiv 2\pi/\lambda$ its wave number.
Perturbations below the \emph{Jeans length}
$$ \lambda_{\rm J} = \frac{ 2\pi c_{\rm s} }{ \sqrt{S_d G\rho_0} } $$
will thus oscillate as as acoustic waves, whereas any disturbance on larger scales will be exponentially amplified.
Recalling that $\beta=\frac{4\pi^2\cs}{S_d\,G\rh L^2}=\frac{\lambda_{\rm J}^2}{L^2}$, Jeans' result would be expressed in the notation of the present paper by saying that a uniform gas cloud with $\beta\le 1$ is linearly unstable.
For larger values of $\beta$, pressure forces are more important in relation to gravity, and \emph{small} perturbations, i.e. $0 < \rhomm-1 \ll 1$, do not grow.

However, this criterion has an important shortcoming: nothing can be said about the evolution of the system when $\rhomm\gg1$.
An equivalent result for equation~\eqref{eq7} would be the following:

\begin{lem}[Spectral analysis]
\label{th_Jeans}
Given the linearized version of equation~\eqref{eq7} with $0<\alpha<2$, small perturbations of the homogeneous state, $\rho_0=1+\delta$ with $\langle\delta\rangle=0$ and $\Vert\delta\Vert_{L^\infty} \ll 1$, are unstable if $\beta\le 1$, and they are damped if $\beta>1$.
\end{lem}
\begin{proof}
From the linearized version of equation~\eqref{eq7}
$$
\pat \delta = -\beta \Lambda^{\alpha}\delta + \delta
$$
the proof follows by taking the Fourier transform.
\end{proof}

An analogous result for equation~\eqref{eq7.1} has been derived by \cite{Chb,Chc}.
The condition for linear (in)stability is indeed the same for the three sets of equations.
In the fully non-linear regime, Theorem~\ref{Max1_d2} proves stability against perturbations of magnitude $\rhomm<c_{\alpha,d}\beta$ (above such threshold, Theorem 1.10 of \cite{LRZ} may apply).
For $d=1$, gravitational collapse cannot occur at all if $1<\alpha<2$ o $\alpha=1$ and $\beta\geq 4\pi^2$ (Theorem~\ref{Max}), the latter condition being equivalent to $\|\rho_0\|_{L^1}\leq \frac{1}{2\pi}$ (see Remark in Section~\ref{th_L1norm}).

\section{The `Darcy approximation'}
\label{sec_Darcy}

As shown in the Introduction, the approximation~\eqref{eq_Darcy} that the velocity is proportional to the acceleration leads immediately to equation~\eqref{eq7.1}.
Here we would like to justify the validity of such approximation and discuss the main similarities and differences with respect to the original problem, as well as other, more traditional approaches.

First, let us note that the family of hydrostatic equilibrium solutions of the full problem,
$$
\nabla P + \rho \nabla U = 0,
$$
are also equilibrium solutions of our approximation, since, taking the divergence of the above expression,
$$
\beta \Delta\rho + \rho(\rho-1) + \nabla\rho \cdot \nabla T(\rho) = 0.
$$

One of these equilibrium solutions is the homogeneous state $\rho(\vv{r},t)=\rh$ that we take as a starting point in the present work.
If we now perturb a `spherical' region of radius $R$ and mass $M=\rh S_d R^d/d$ by a small displacement $r=R[1+\epsilon(R)]$ with $\epsilon(R)\ll 1$ for all $R$, and the perturbed density decreases monotonically with radius\footnote{This will always be the case in the vicinity of a density peak, and, for Gaussian random fluctuations, it will also hold, on average, at large radii \cite{Bardeen+86}.}, the different spherical shells will not cross during the collapse phase, and the enclosed mass $M(r)$ will be conserved for every shell.
Applying Gauss theorem to equation~\eqref{eq_Poisson}, their evolution will be given by
$$ S_d\, r^{d-1}\, \nabla U(r) = S_d\, G \left( M - \rh S_d\,\frac{r^d}{d}\, \right) $$
and therefore, neglecting pressure forces (by imposing that the initial displacement $\epsilon$ is sufficiently smooth),
$$
\ddot r = - \nabla U(r)
= \frac{ G S_d\,\rh }{ d } \left[ 1- \left(R/r\right)^d \right]
= \frac{ G S_d\,\rh }{ d } \left[ 1- \frac{1}{(1+\epsilon)^d} \right].
$$
To first order, $\ddot\epsilon = G S_d\,\rh \epsilon$, and we recover expression~\eqref{eq_Darcy} with a characteristic dynamical time
\begin{equation}
\label{eq_tau}
\tau \equiv \frac{ \dot\epsilon }{ \ddot\epsilon }
= \frac{ 1 }{ \sqrt{ S_d\,G\rh } }.
\end{equation}

The `Darcy approximation' consists in assuming that this relation, strictly valid close to equilibrium, holds at all times.
As pointed out in~\cite{Chb}, the main difference between the Euler-Poisson and the Smoluchowski / Patlak-Keller-Segel model is that, in the latter, perturbations below the Jeans length do not oscillate as sound waves but are exponentially damped.

The extrapolation of Equation~\eqref{eq_Darcy} beyond the linear regime is very similar in spirit (though not exactly equivalent) to the Zel'dovich approximation \cite{Zeldovich70} used in Cosmology to study the formation of large-scale structure in the primordial universe.
In this approximation, particles move along straight trajectories of the form
$$
\vv{r}(t) = a(t)\, \vv{q} + b(t)\, \vv{p}(\vv{q})
$$
where $\vv{r}$ denotes the actual (Eulerian) position of the particle, $\vv{q}$ is its Lagrangian coordinate, the factor $a(t)$ accounts for cosmic expansion, and the vector function $\vv{p}(\vv{q})$ is determined by the initial conditions.
The evolution of the inhomogeneities is given by the growth factor $b(t)$, chosen to match the results of Eulerian linear theory.
For an Einstein-deSitter cosmology, which always provides a valid approximation for a matter-dominated universe at early times, $b = a^2$, $\left( \frac{ \dot{a} }{ a } \right)^2 = \frac{ 8\pi G }{ 3 } \rh$, and $\frac{ \ddot{a} }{ a } = \frac{ 4\pi G }{ 3 } \rh$.
Working in comoving coordinates, $\vv{x}=\vv{r}/a$, one obtains
$$
\frac{ \dot{x} }{ \ddot{x} } = \frac{ \dot{a} }{ \ddot{a} } = \sqrt{ \frac{ 3 }{ 2 \pi G \rh } }
$$
analogously to equation~\eqref{eq_tau}, although in this case $\rh$ depends on time due to the expansion of space (see \cite{BuchertDominguez05} for a detailed discussion of the cosmological implications of the parallelism between peculiar velocities and accelerations).

Again, the most important difference with respect to the Euler equation is that~\eqref{eq7.1} is dissipative.
However, we would like to note that, in order to prevent caustic formation and shell-crossing in the Zel'dovich approximation, a small dissipative viscosity is introduced, leading to the so-called adhesion model \cite{Gu89,Gu11}.
From a microphysical point of view, this viscosity may arise from random noise in the trajectories of the particles.
It has been shown \cite{Chd,Che} that equation~\eqref{eq7.1} can be obtained for a gas of self-gravitating Brownian particles in the mean field limit with strong friction.
Along these lines, our generalized equation~\eqref{eq7} would arise if the random perturbation (due, for instance, to two-body encounters between particles) was described by a L\'evy noise.

\vspace{1cm}

\end{document}